\documentclass[a4paper,11pt]{article}

\usepackage[ngerman, american]{babel} 
\usepackage[utf8]{inputenc}
\usepackage[a4paper, left=3cm, right=3cm, top=2.5cm, bottom=2.5cm]{geometry}
\usepackage{amssymb,amsmath,amsthm}
\usepackage{mathrsfs}
\usepackage{dsfont}

\usepackage{graphicx}
\usepackage{caption}
\usepackage{color}
\usepackage{tikz}
\usepackage[numbers]{natbib}

\numberwithin{equation}{section} 
\usepackage{comment}
\usepackage{hyperref}

\newcommand{\eps}{{\varepsilon}}  
          
\newcommand{\vphi}{\varphi}

\newcommand{\NN}{\nonumber}

\newcommand{\cC}{\mathcal{C}}

\newcommand{\cF}{\mathcal{F}}

\newcommand{\cN}{\mathcal{N}}         
\newcommand{\cO}{\mathcal{O}}

\newcommand{\cR}{\mathcal{R}}
\newcommand{\cS}{\mathcal{S}}

\newcommand{\N}{\mathds{N}}

\newcommand{\Q}{\mathds{Q}}
\newcommand{\R}{\mathds{R}}
\newcommand{\C}{\mathds{C}}

\newcommand{\re}{{\mathrm e}}
\newcommand{\ri}{{\mathrm i}}
\newcommand{\dx}{{\mathrm d}}
\newcommand{\I}{\mathds{1}}

\renewcommand{\P}{\mathds{P}}
\newcommand{\E}{\mathds{E}}
\newcommand{\Var}{\mathrm{Var}}

\newcommand{\Tr}{\mathop{\mathrm{Tr}}}
\newcommand{\Id}{\mathrm{Id}}

\newcommand{\rem}{\mathfrak{r}}
\newcommand{\bfX}[1]{\smash{\mathbf{X}_{#1}^{(N)}}}
\newcommand{\rate}{\mathcal{I}}

\newtheorem{theorem}{Theorem}[section]         
\newtheorem{lemma}[theorem]{Lemma}

\newtheorem{proposition}[theorem]{Proposition}
\newtheorem{assumption}[theorem]{Assumption}

\newtheoremstyle{myrem}
  {}
  {}
  {}
  {}
  {\bfseries}
  {.}
  { }
  {\thmname{#1}\thmnumber{ #2}\thmnote{\normalfont{ (#3)}}}
\theoremstyle{myrem}
\newtheorem*{remark}{Remark}

\begin{document}
\title{\vspace{-2cm} LDP for the largest eigenvalue of Kronecker random matrices}
\author{Alice Guionnet\thanks{UMPA, ENS de Lyon, France. E-mail: alice.guionnet@ens-lyon.fr.} \and Jonathan Husson\thanks{LMBP, Université Clermont Auvergne, France. E-mail: jonathan.husson@uca.fr.} \and Jana Reker\thanks{UMPA, ENS de Lyon, France. E-mail: jana.reker@ens-lyon.fr.}}

\maketitle

\begin{abstract}
We prove a large deviations principle for the largest eigenvalue of Gaussian Kronecker matrices, namely matrices defined as the sum of tensors of independent Gaussian matrices in the regime where the dimension of the Gaussian matrices goes to infinity. 
\end{abstract}


\section{Introduction}
The theory of random matrices has developed significantly since the seminal works of Wishart \cite{Wishart1928} and Wigner~\cite{Wigner1958}, with applications across many areas of mathematics, physics, and statistics. In this article, we focus on their large deviations theory, namely on the question of estimating the probability that their spectrum has an unlikely behavior. Such a theory has several motivations. Recently, large deviations for the smallest eigenvalue of Gaussian matrices have emerged as a key tool for estimating the mean number of minima of complex Gaussian functions via the Kac-Rice formula \cite{AuffingerBC13,AuffingerB13,BenarousMMN19}, an important step in analyzing the complexity of spin glasses. But since Boltzmann, the theory of large deviations has also been a natural approach to defining entropy. 

\medskip
This path was followed by Voiculescu to define entropy in free probability, which can be thought of as a non-commutative probability theory endowed with a notion of freeness, analogous to independence in classical probability theory. In the key article \cite{voiculescu91}, he showed that unitarily invariant matrices converge as the dimension goes to infinity towards free variables, in the sense that the traces of their words converge towards the trace of words in free variables. Voiculescu then introduced an entropy theory \cite{Voiculescu93,Voiculescu94}. He defined a microstate entropy in the spirit of Boltzmann by the volume of independent Gaussian matrices whose empirical non-commutative distribution approximates a given tracial state, namely the large deviation rate function for the non-commutative distribution of independent Gaussian matrices. One of the objectives of this theory was to define invariants for free group factors, but unfortunately, as pointed out by Voiculescu~\cite[Sect.~2.6]{Voiculescu02}, technical difficulties still prevent the completion of the theory. In particular, a full large deviation principle is still lacking, which could take the form of the convergence of the renormalized logarithm of the volume of microstates (so far, microstates free entropy is instead defined in terms of a limsup) or more boldly of an equality between the microstate entropy and the microstate free entropy which were so far only shown to be bounded by one another~\cite{BianeCG03,Jekel20}.

\medskip
In~\cite{Voiculescu02}, he defined another entropy, called topological entropy, which, instead of considering the large deviations of the traces of polynomials in independent matrices, considers the large deviations of their norm. However, so far this definition has also suffered from the difficulties encountered with the microstates entropy. In ergodic theory, Bowen~\cite{Bowen08,Bowen10} defined an entropy theory for group actions inspired by Kolmogorov and Sinai. Again, this theory is incomplete due to technical difficulties, and a full large deviation principle for sofic entropy is not yet known. 
More recently, Austin \cite{Austin-split1, Austin-split2} defined a new notion of entropy for unitary representations, roughly akin to a microstate entropy for the types of these representations. Interestingly, this new theory allows proving a full large deviation principle for the operator norm of a polynomial in independent unitary matrices~\cite[Thm.~C]{Austin-split2}. 

\medskip
Large deviations for a single random matrix are much better understood. It was first analyzed for the so-called Gaussian ensembles, which are self-adjoint matrices with i.i.d. centered Gaussian entries above the diagonal. In fact, the joint law of the eigenvalues of such matrices is explicit and given by a Coulomb gas interaction. From this explicit formula and refined Laplace's method arguments, a large deviation principle for the empirical measure of the eigenvalues was derived in~\cite{BenarousG97}, 
as well as a large deviation for the largest eigenvalue~\cite{BenarousDG01}. Removing the assumption of Gaussian entries is still an open question as far as it concerns the large deviations of the empirical measure of the eigenvalues, and that the entries have sub-Gaussian tails. They are now well understood for the extreme eigenvalues~\cite{GuionnetHusson2020,AGH2021,CookDG23} but their derivation requires sophisticated tools such as tilting the measure by spherical integrals. Indeed, in such cases, the joint law of the eigenvalues is not explicit, and large deviations are approached by deforming the measures to make the desired deviation likely similar to the classical proof of Cramér's theorem. If the entries have tails heavier than Gaussian (but stretched exponentials), it turns out that large deviations have a smaller rate and can be derived both for the empirical measure~\cite{BordenaveC14} and the extreme eigenvalues~\cite{Augeri16}. Unlikely events are, in this case, created by anomalously large entries. For heavy tails such as $\alpha$-stable laws or sparse matrices given by Hadamard products with the adjacency matrix of an Erd\"os-R\'enyi graph, large deviations can also be studied \cite{BordenaveGM24,GangulyHN24}.

\medskip
Understanding Gaussian matrices with additional structure is also an important topic, particularly for its application to the study of the complexity of random functions (see, e.g.,~\cite{AuffingerB13}). The first case studied was Gaussian matrices with a variance profile: the entries are independent centered Gaussian variables but their variances vary. Again, in this case, the joint law of the eigenvalues is not explicit, but the trick of tilting the measure by spherical integrals could be used to derive large deviations for the largest eigenvalue \cite{Husson2022,HussonMcKenna2023,DGH2024}. The large deviations principle for the empirical measure has not yet been demonstrated in this case (see~\cite{Guionnet02} for upper bounds). Another case of interest is the deformed Gaussian matrix, i.e., the sum of a Gaussian and a deterministic matrix. Using Dyson's Brownian motion, the large deviations principle for the empirical measure of such a matrix was derived in~\cite{GuionnetZ02,GuionnetH23}. The large deviation for the largest eigenvalue was established in~\cite{Mckenna21} when the deterministic matrix has no outliers and in~\cite{BoursierG24} when it has one outlier. In the latter case, large deviations are deduced from a new argument based on a functional equation satisfied by the rate function. 

\medskip
In this article, we focus on matrices that are sums of tensor products of independent self-adjoint matrices with fixed-size deterministic matrices (see \eqref{eq-model} below for the precise definition of the model). Random matrix models of this structure are referred to as (Gaussian) Kronecker matrices in the literature~(cf.~\cite{AEKN2019}). They became particularly popular in free probability since they were shown to encapsulate all the information of polynomials in several matrices, thanks to the so-called linearization trick \cite{HaagerupTobjornsen2005,AGZBook}. Furthermore, they provide a simple model of a structured matrix with dependent entries. The main result of this paper is a large deviation principle for the largest eigenvalue of Kronecker matrices. It is to our knowledge the first large deviations principle for matrices with correlated entries. Moreover, the model involves a structured deformation that goes beyond the diagonal deformations considered in~\cite{HussonMcKenna2023}. Since there are well-known links between such Kronecker matrices and polynomials in several random matrices through Haagerup and Thorbj\o rnsen's so-called linearization trick (see~\cite{HaagerupTobjornsen2005, AGZBook}), in future work we hope to apply our findings to tackle large deviations for the norm of such polynomials and provide new light on Voiculescu's topological entropy. One should notice that recently a similar result was announced by T. Austin in~\cite[Theorem~C]{Austin-split2} for words in unitary matrices. 

\medskip
\textbf{Acknowledgements:} A.G. and J.R. were partially supported by ERC Advanced Grant ``LDRaM'' No. 884584.

\subsection{Definition of the Model}
We consider $NL\times NL$ random matrices of the form
\begin{equation}\label{eq-model}
\mathbf{X}_\beta^{(N)}:=\sum_{j=1}^k A_j\otimes W_j^{(N)}+A_0\otimes \Id_N,
\end{equation}
where $A_0,\dots,A_k$ are $L\times L$ deterministic Hermitian matrices, $\Id_N$ denotes the $N\times N$ identity matrix, and $\smash{W_1^{(N)},\dots,W_k^{(N)}}$ are i.i.d. centered $N\times N$ Hermitian random matrices. For simplicity, we drop the $N$-dependence of $\smash{W_1^{(N)},\dots,W_k^{(N)}}$ throughout the paper. The parameter~$\beta$ distinguishes between a real symmetric ($\beta=1$) and a complex Hermitian ($\beta=2$) variant of the model~\eqref{eq-model} for which we make the following assumptions.

\begin{assumption}[Assumptions on matrices]\label{assu-matrix}
Fix $\beta\in\{1,2\}$.
\begin{itemize}
\item[$\beta=1$] We pick symmetric $A_0,\dots,A_k\in\R^{L\times L}$ and the random matrices $W_1,\dots,W_k$ are sampled independently from the Gaussian Orthogonal Ensemble (GOE), i.e., $W_1=W_1^T$ and $(W_1)_{ij}$, $i\leq j$, are independent real centered Gaussian random variables with variance $(1+\I_{\{i=j\}})/N$.
\item[$\beta=2$] We pick Hermitian $A_0,\dots,A_k\in\C^{L\times L}$ and the random matrices $W_1,\dots,W_k$ are sampled independently from the Gaussian Unitary Ensemble (GUE), i.e., $W_1=W_1^*$, where $(W_1)_{ij}$, $i<j$, are independent complex centered Gaussian random variables with variance $1/N$ (meaning that their real and imaginary parts are independent Gaussians of variance $1/(2N)$), and $(W_1)_{jj}$ are independent real centered Gaussian random variables with variance $1/N$ that are further independent from $(W_1)_{ij}$, $i<j$.
\end{itemize}
\end{assumption}

We remark that $\bfX{\beta}$ is a correlated random matrix with the correlation structure of each summand prescribed by the deterministic matrices $A_1,\dots,A_k$. Due to the Kronecker product in the definition, the model~\eqref{eq-model} is also called a \emph{(Hermitian) Kronecker random matrix} (see~\cite{AEKN2019}). In this context, the deterministic matrices $A_1,\dots,A_k$ are referred to as \emph{structure matrices}. We further remark that~$\E\bfX{\beta}=A_0\otimes\Id_N$, i.e.,~\eqref{eq-model} is a deformed model and its expectation is governed by $A_0$. We denote the eigenvalues of a matrix $T$ by $\lambda_1(T)\geq\dots\geq\lambda_N(T)$. Lastly, we remark that $L$ is independent of the dimension and therefore the matrices $A_j$ are uniformly bounded. 

\medskip
We conclude the discussion of the model by noting some general properties of~\eqref{eq-model} that are used throughout the analysis. First, let $\mu_N$ denote the empirical spectral measure of~$\bfX{\beta}$, i.e., the probability measure on the real line given by 
$$\mu_N=\frac{1}{N}\sum_{i=1}^N\delta_{\lambda_i(\bfX{\beta})}.$$
By the global law in~\cite[Thm. 2.7]{AEKN2019}, there exist deterministic probability measures $\nu_N$ such that $(\mu_N-\nu_N)\rightarrow0$ weakly in probability. The deterministic approximation $\nu_N$ is characterized by the Matrix Dyson Equation (MDE)
\begin{equation}\label{eq-MDE}
-\frac{1}{\mathbf{M}(z)}=z\mathbf{Id}_{LN}-A_0\otimes\I_N+\mathscr{S}[\mathbf{M}(z)],\quad\Im z>0,
\end{equation}
where $\mathbf{M}$ is chosen such that $\Im \mathbf{M}$ is positive definite and
\begin{displaymath}
\mathscr{S}[\mathbf{T}]=\E[\mathbf{X}_\beta^{(N)} \mathbf{TX}_\beta^{(N)}],\quad \mathbf{T}\in\C^{LN\times LN}.
\end{displaymath}
Note that~\eqref{eq-MDE} is an $LN\times LN$ matrix equation. However, observe that setting $\mathbf{T}=T\otimes \Id_N$ yields
\begin{displaymath}
\mathscr{S}[T\otimes \Id_N]=\Big(\sum_{j=1}^kA_jTA_j\Big)\otimes \Id_N.
\end{displaymath}
We can hence exploit the tensor product structure of our model to make the ansatz $\mathbf{M}(z)=M(z)\otimes \Id_N$, which reduces~\eqref{eq-MDE} to the $L\times L$ matrix equation
\begin{equation}\label{eq-tensorMDE}
-\frac{1}{M(z)}=z\Id_L-A_0+\cS[M(z)],\quad \Im z>0,
\end{equation}
where $\cS[T]=\sum_{j=1}^kA_jTA_j$ for $T\in\C^{L\times L}$. This is a purely deterministic equation that only depends on the $L\times L$ matrices $A_1,\dots, A_k$. We further remark that $M(z)$ admits a Stieltjes transform representation of the form
\begin{equation}\label{eq-Stieltjestransf}
M_{ij}(z)=\int_{\R}\frac{\mu_{ij}(\dx x)}{x-z},
\end{equation}
where $\mu_{\mathrm{MDE}}(\dx x)=(\mu_{ij}(\dx x))_{i,j=1}^L$ is a (uniquely defined) positive semi-definite matrix-valued measure on $\R$ with normalization $\mu_{\mathrm{MDE}}(\R)=\Id_L$ (see, e.g.,~\cite[Prop.~2.1]{AEK2019}). The density of the limiting spectral measure is then obtained by Stieltjes inversion of $\langle M(z)\rangle=\Tr[M(z)]/L$. In particular, we have $\mu_\infty=(\mu_{11}+\dots+\mu_{LL})/L$.

\medskip
The object of interest of our analysis is the largest eigenvalue $\lambda_1$ of $\bfX{\beta}$. Under Assumption~\ref{assu-matrix}, the model~\eqref{eq-model} is a mean-field model and the largest eigenvalue $\lambda_1$ will converge towards the right end of the support of $\mu_\infty$. More precisely, we obtain the following property from~\cite[Thm.~4.7]{AEKN2019}.

\begin{lemma}\label{lem-convlargestev}
Fix $\beta\in\{1,2\}$. For every $\delta>0$,
\begin{displaymath}
\lim_{N\rightarrow\infty}\P(|\lambda_1(\mathbf{X}_\beta^{(N)})-r_\infty |\geq\delta)=0,
\end{displaymath}
where $r_\infty=\max\mathrm{supp}(\mu_\infty)$.
\end{lemma}

{In general, the exact value of $r_\infty$ is only available implicitly from identifying the limiting spectral measure through its Stieltjes transform. However, for Gaussian models such as~\eqref{eq-model}, a formula for $r_\infty$ in the form of a variational principle is available (see~\cite{Lehner1999}).}

\subsection{Statement of the Main Result}
The main result of this paper is a large deviation principle for the largest eigenvalue of~$\bfX{\beta}$. Before stating it, let us first introduce the definitions that are needed to formulate the theorem. We give a sketch of the proof in Section~\ref{sect-manual} below to link these quantities to~\eqref{eq-model} and give an intuition for the form of the rate function.

\medskip
We recall that the Stieltjes transform $m_\nu$ of a probability measure $\nu$ is given by
\begin{equation}\label{eq-Stieltjesconv}
m_\nu(z)=\int\frac{\nu(\dx x)}{x-z}\,.
\end{equation}
If $r_\nu$ denotes the rightmost point in the support of $\nu$, $m_\nu$ is strictly increasing on $(r_\nu,+\infty)$ with values in $(\lim_{y\downarrow r_\nu} m_\nu(y),0)$ and we denote 
by $(-m_\nu)^{-1}$ the functional inverse of its negative from $(\lim_{y\downarrow r_\nu} m_\nu(y),0)$ into $(r_\nu,\infty)$. For $\theta\geq0$, define the function
\begin{align}
&J_\nu(x,\theta)\label{eq-defJ}\\
&:=\begin{cases}\theta (-m_\nu)^{-1}(2\theta)-\frac12(1+\ln(2\theta))-\frac12\int_\R\ln\big((-m_\nu)^{-1}(2\theta)-y\big)\dx\nu(y), &\text{if }2\theta\le-m_\nu({x}),\\
\theta x-\frac12(1+\ln(2\theta))-\frac12\int_\R \ln|x-y|\dx\nu(y),&\text{if }2\theta\ge -m_\nu({x}).\\
\end{cases}\NN
\end{align}
Note that~\eqref{eq-defJ} differs slightly from the function $J_\nu$ defined in~\cite{GuionnetMaida2005,DGH2024} due to a different convention for the Stieltjes transform in~\eqref{eq-Stieltjesconv}.
\medskip

Next, for $\Psi$ in the space $\mathrm{Sym}^{+,1}_L(\R)$ of positive semi-definite matrix {so that $\Tr(\Psi)=1$}, let
\begin{equation}\label{eq-defK}
K(\theta,\Psi):=L^2\theta^2\Tr[\Psi^T\cS(\Psi^T)]+L\theta\Tr[A_0^T\Psi]+\frac{1}{2}(\ln(\det(\Psi))-L\ln(L)),
\end{equation}
with $\cS$ as introduced in~\eqref{eq-tensorMDE}. We further define the matrices
\begin{align}
\vphi(\theta,x)&:=-\frac{M(\max\{(-m_{\mu_\infty})^{-1}(2\theta),x\})}{2\theta L},\label{eq-defvphi}\\
\phi(\theta,x,\Psi)&:=\vphi(\theta,x)+\Big(1+\frac{m_{\mu_\infty}(x)}{2\theta}\Big)_+\Psi,\label{eq-defphi}
\end{align}
where $M$ is the solution to~\eqref{eq-tensorMDE}. Recall that $\mu_\infty$ denotes the limiting spectral measure for~\eqref{eq-model}.

\medskip
Lastly, combine the above quantities to obtain
\begin{equation}\label{eq-defF}
\cF(\theta,x,\Psi,\beta):=\beta\big( LJ_{\mu_\infty}(x,\theta)-K(\theta,\phi(\theta,x,\Psi))\big).
\end{equation}

With these definitions in place, we can now state our main result. To make the differences between the real symmetric case ($\beta=1$) and the complex Hermitian case ($\beta=2$) more visible, we state them separately. A discussion of the properties of the rate function is included in Section~\ref{sect-ratefunct} below (see~Proposition~\ref{prop-ratefunct}).


\begin{theorem}\label{thm-main}
Let $\beta=1$ and construct~\eqref{eq-model} such that Assumption~\ref{assu-matrix} is  satisfied. Then the law of $\lambda_1(\mathbf{X}_1^{(N)})$ satisfies a large deviation principle with speed $N$ and a good rate function $\rate_1$. The function $\rate_1(x)$ is infinite for $x<r_\infty$, satisfies $\rate_1(r_\infty)=0$, and for $x>r_\infty$, we have
\begin{align}
\rate_1(x)=\inf_{\substack{\Psi\in\mathrm{Sym}_L^{{+,1}}(\R)\\ \Tr[\Psi^T\cS(\Psi)]\neq0}}\sup_{\theta\geq0}\cF(\theta,x,\Psi,1),\label{eq-ratefunct}
\end{align}
where $\cF$ was defined in~\eqref{eq-defF}.
\end{theorem}

\begin{theorem}\label{thm-main2}
The LDP in Theorem~\ref{thm-main} continues to hold for $\beta=2$ if we replace~\eqref{eq-ratefunct} by
\begin{align}
\rate_2(x)=\inf_{\substack{\Psi\in\mathrm{Sym}_L^{{+,1}}+(\C)\\ \Tr[\Psi^*\cS(\Psi)]\neq0}}\sup_{\theta\geq0}\cF(\theta,x,\Psi,2),\label{eq-ratefunct2}
\end{align}
where $\cF$ was defined in~\eqref{eq-defF}.
\end{theorem}
Note that real symmetric $A_0,\dots,A_k$ satisfy Assumption~\ref{assu-matrix} for both $\beta=1$ and $\beta=2$, allowing to compare the resulting rate functions~\eqref{eq-ratefunct} and~\eqref{eq-ratefunct2} in this case. We readily obtain the bound $\rate_2(x)\leq2\rate_1(x)$.

\medskip
The remainder of the article is dedicated to the proof of Theorems~\ref{thm-main} and~\ref{thm-main2}, and the discussion of the rate function. For a concise presentation of the argument, we focus on the proof of Theorem~\ref{thm-main} (real case) and discuss the necessary modifications for the proof of Theorem~\ref{thm-main2} (complex case) separately in Section~\ref{sect-GUE}. Let us, therefore, fix $\beta=1$ for now and drop the $\beta$-dependence from $\smash{\mathbf{X}^{(N)}_\beta}$, {${\mathcal I}_\beta$, ${\mathcal F}(\theta,x,\Psi,\beta)$}, etc. to simplify notation.

\subsection{General Notation}
We conclude the introduction with some additional notations and conventions that are used throughout the paper. Firstly, we always use capital letters to denote matrices and lowercase letters to denote vectors and scalars. To make the dimension of a matrix resp. vector more visible, we further use boldface to denote the $LN\times LN$ tensor products while its factors of size $L\times L$ resp. $N\times N$ are written non-bold.

\medskip
The identity matrix of size $N\times N$ is denoted by $\Id_N$ and $e_j$ denotes the unit vector with the $j$-th entry equal to one. For a matrix $A$, the transpose and Hermitian conjugate are denoted by $A^T$ and $A^*$, respectively, and we write $\overline{z}$ for the complex conjugate of a scalar $z\in\C$. Further, $(A-z\Id_N)^{-1}$ denotes the resolvent of $A\in\C^{N\times N}$ at~$z$. Note that the sign convention matches~\eqref{eq-Stieltjesconv} in the sense that $\Tr[(\bfX{\beta} -z \Id_N)^{-1}]/N=m_{\mu_N}(z)$. Our default norm~$\|\cdot\|$ is the operator norm for matrices resp. the Euclidean norm for vectors, and we set $\|v\|_\infty=\max_i|v_i|$ for $v\in\C^N$ resp. $\|A\|_\infty=\max_{i,j}|A_{i,j}|$ for $A\in\C^{N\times N}$. Moreover, we use $\langle\cdot,\cdot\rangle$ for the Frobenius inner product for matrices resp. the Euclidean inner product for vectors. Throughout the paper, we follow the convention that complex inner products are skew-linear in the first entry. 

\medskip
Given $n\in\N$, we abbreviate $[n]=\{1,\dots,n\}$. The eigenvalues and eigenvectors of a matrix $A\in\C^{N\times N}$ are denoted by $\lambda_j(A)$ resp. $v_j(A)$ for  $j\in[N]$, and we omit the matrix whenever the association is clear. Throughout the analysis, we implicitly assume the eigenvalues ordered as $\lambda_1\geq\dots\geq\lambda_N$ and the eigenvectors to be normalized to $\|v_j\|=1$.  Moreover, $\mathrm{Sym}_N(\R)$ and $\mathrm{Sym}_N(\C)$ denote the set of symmetric matrices in $\R^{N\times N}$ resp. the set of Hermitian matrices in $\C^{N\times N}$. We further add a superscript $\mathrm{Sym}_N^{++}(\R)\subset\R^{N\times N}$ to denote the set of symmetric positive definite matrices or
\begin{displaymath}
\mathrm{Sym}_N^+(\R)=\overline{\mathrm{Sym}_N^{++}(\R)}
\end{displaymath}
for the set of positive semi-definite matrices. The sets $\mathrm{Sym}_N^{++}(\C)$ and $\mathrm{Sym}_N^{+}(\C)$ are defined analogously. We also set
\begin{displaymath}
\mathrm{Sym}_L^{\sharp,1}(\R)=\{X\in\mathrm{Sym}_L^\sharp(\R):\Tr(X)=1\}
\end{displaymath}
with $\sharp\in\{+,++\}$ to denote the matrices with trace one. Again, $\mathrm{Sym}_L^{\sharp,1}(\C)$ is defined analogously.

\medskip
Lastly, $\rem$ denotes a generic error term that vanishes when a certain limit is taken. We will specify the variables that $\rem$ depends on and the order of limits for each occurrence. {Moreover, $o_a(1)$ (resp. $o(a)$) denotes a quantity going zero  (resp. so that $o(a)/a$ goes to zero)  when some specified limit is taken on $a$, whereas $\cO(X)$ is a function so that there exists a finite positive constant $C$ such that $|\cO(X)/X|\in [C^{-1},C]$.}
\section{Outline of the Argument}\label{sect-manual}
The aim of this section is to provide a general road map to the proof of Theorem~\ref{thm-main}. Following the usual method, our main result is obtained from a weak LDP and the exponential tightness of the associated sequence of probability measures. Noting that the latter is obtained from a standard net argument (see Lemma~\ref{lem-norm_exptight} below), our main focus lies on establishing a weak LDP for $\lambda_1(\mathbf{X}^{(N)})$. Introducing the notation

\begin{equation}\label{eq-defPNdelta}
P_\delta^{(N)}(x):=\frac{1}{N}\ln\big(\P(\{|\lambda_1(\mathbf{X}^{(N)})-x|\leq\delta\})\big),
\end{equation}
it can be stated as follows.

\begin{theorem}[Weak LDP]\label{thm-mainweak}
For $x>{r_\infty}$, it holds that
\begin{displaymath}
\lim_{\delta\rightarrow0}\liminf_{N\rightarrow\infty}P_\delta^{(N)}(x)=\lim_{\delta\rightarrow0}\limsup_{N\rightarrow\infty}P_\delta^{(N)}(x)=-\rate(x).
\end{displaymath}
\end{theorem}
We split the proof of Theorem~\ref{thm-mainweak} into the upper and lower bound, which are addressed in Sections~\ref{sect-wLDPsup} and~\ref{sect-wLDPinf}, respectively.

\medskip
The main ingredient for both bounds is the strategy of \emph{tilting by spherical integrals} that was developed in~\cite{GuionnetHusson2020, Husson2022, HussonMcKenna2023, DGH2024}. The idea behind this technique is similar to the proof of Cramer's theorem (see, e.g.,~\cite[Thm.~I.4]{denHollander2000}), where a suitable shift of the probability measure makes the desired large deviations more likely. Instead of the exponential tilting used in Cramer's theorem, however, it is sufficient to shift the measure in a random direction here. This results in the so-called spherical integral entering the argument which can be estimated using suitable bounds from the literature (see~\cite{GuionnetMaida2005, Maida2007}).

\medskip
For a real number $\theta$ and a symmetric $N\times N$ matrix $H_N$, we define the \emph{spherical integral} by
\begin{equation}\label{eq-spherint}
I_N(H_N,\theta):=\int_{\mathds{S}^{N-1}}\re^{\theta N\langle u,H_Nu\rangle}\dx u=\E_u[\re^{\theta N\langle u,H_Nu\rangle}],
\end{equation}
where the integral is taken with respect to the uniform law on the $N$-dimensional sphere~$\mathds{S}^{N-1}$. It was shown in~\cite{GuionnetMaida2005, Maida2007,GuionnetH22} that, given a suitable {sequence of matrices $(H_N)_N$}, the spherical integral~\eqref{eq-spherint} converges as $N\rightarrow\infty$.

\begin{lemma}[{\cite[Prop.~1]{GuionnetH22}}]\label{lem-convsphint}
Let $(H_N)_N$ be a sequence of $N\times N$ self-adjoint deterministic matrices such that $\sup_{N}\|H_N\|<\infty$, the empirical eigenvalue measures $(\mu_{H_N})_N$ converge weakly towards a probability measure $\nu$, and the largest eigenvalues $(\lambda_1(H_N))_N$ converge to some number $\lambda_{\mathrm{max}}\in\R$. Then
\begin{equation}\label{eq-convsphint}
\lim_{N\rightarrow\infty}\frac{1}{N}\ln[I_N(H_N,\theta)]=J_{\nu}(\lambda_{\mathrm{max}},\theta),
\end{equation}
where $J_\nu(y,\theta)$ is given in~\eqref{eq-defJ}. Moreover, the convergence in~\eqref{eq-convsphint} is uniform on small neighborhoods of $\mu_{H_N}$ (for the weak topology) and $\lambda_1(H_N)$.
\end{lemma}

In the proof of Theorem~\ref{thm-mainweak}, we encounter~\eqref{eq-spherint} as $I_N(\mathbf{X}^{(N)},\theta)$, where $\mathbf{X}^{(N)}$ is the model~\eqref{eq-model}. Note that, in our setting, the weak convergence of $(\mu_N)_N$ and Lemma~\ref{lem-convlargestev} imply that Lemma~\ref{lem-convsphint} applies point-wise on a suitable set of high probability. Here, the limit is given by $J_{\mu_\infty}(r_\infty,\theta)$ with $J$ as in~\eqref{eq-defJ}, $\mu_\infty$ being the weak limit of the $(\mu_N)_N$, and $r_\infty$ as in Lemma~\ref{lem-convlargestev}. Observe that this resembles the first summand of~\eqref{eq-defF} which appears the rate function of the LDP in Theorem~\ref{thm-main}.

\medskip
A key trick for dealing with the randomness in $I_N(\mathbf{X}^{(N)},\theta)$ is an application of Fubini's theorem of the form
\begin{equation}\label{eq-Fubinitrick}
\E_W\E_{\mathbf{u}}[\I_S\, \re^{\theta NL\langle \mathbf{u},\mathbf{X}^{(N)}\mathbf{u}\rangle}]=\E_{\mathbf{u}}\E_W[\I_S\,\re^{\theta NL\langle \mathbf{u},\mathbf{X}^{(N)}\mathbf{u}\rangle}]
\end{equation}
for some suitable event $S$. Here, $\E_W$ denotes expectation w.r.t. the randomness in the matrices $W_1,\dots,W_k$. The quantity on the right-hand side of~\eqref{eq-Fubinitrick} is called the \emph{annealed spherical integral}. We consider its convergence next to obtain the second function $K$ that appears in~\eqref{eq-defF}. To formulate the result, we first introduce some notation. Let~$\mathbf{u}$ be a random vector that is uniformly distributed on the unit sphere $\mathds{S}^{NL-1}$ and write $\mathbf{u}=(u_1,\dots,u_L)$ with blocks $u_1,\dots,u_L$ of length $N$. It holds that
\begin{displaymath}
\mathbf{u}\overset{law}{=}\frac{(g_1,\dots,g_L)}{\sqrt{\sum_{l=1}^L\|g_j\|^2}},
\end{displaymath}
where $g_1,\dots,g_L$ are i.i.d. random (row) vectors distributed according to $\mathcal{N}(0,\mathrm{Id}_N)$. Moreover, we have 
\begin{equation}\label{eq-uinnprod}
\langle u_i, u_j\rangle\overset{law}{=}\frac{\langle g_i,g_j\rangle}{\sum_{s=1}^L\|g_s\|^2}.
\end{equation}
Let $G\in\R^{L\times N}$ be the Gaussian random matrix whose rows are given by $g_1,\dots,g_L$ and denote by $Y=GG^T$ the real Wishart matrix constructed from it. As $Y_{ij}=\langle g_i,g_j\rangle$ and thus $Y_{jj}=\|g_j\|^2$, we reobtain the inner products in~\eqref{eq-uinnprod} by introducing
\begin{equation}\label{eq-defYtilde}
{\widetilde{Y}_{ij}:=\langle u_i,u_j\rangle\overset{law}{=}\frac{Y_{ij}}{\Tr{Y}}.}
\end{equation}
We can think of $\widetilde{Y}$ as a renormalized Wishart matrix. In this notation, the convergence of the annealed spherical integral can be stated as follows.

\begin{proposition}\label{prop-conv.annealed}
For $\widetilde{Y}$ as in~\eqref{eq-defYtilde} and $\Psi\in\mathrm{Sym}^{+,1}_L(\R)$, it holds that
\begin{align*}
K(\theta,\Psi)&=\lim_{\delta\rightarrow0}\limsup_{N\rightarrow\infty}\frac{1}{N}\ln\Big(\E_{\mathbf{u}}\E_W\Big[\re^{\theta NL\langle \mathbf{u},\mathbf{X}^{(N)}\mathbf{u}\rangle}\I_{\{\|\widetilde{Y}-\Psi\|\leq\delta\}}\Big]\Big)\\
&=\lim_{\delta\rightarrow0}\liminf_{N\rightarrow\infty}\frac{1}{N}\ln\Big(\E_{\mathbf{u}}\E_W\Big[\re^{\theta NL\langle \mathbf{u},\mathbf{X}^{(N)}\mathbf{u}\rangle}\I_{\{\|\widetilde{Y}-\Psi\|\leq\delta\}}\Big]\Big),
\end{align*}
where $K(\theta,\Psi)$ was defined in~\eqref{eq-defK}.
\end{proposition}
We give the proof of Proposition~\ref{prop-conv.annealed} in Section~\ref{sect-conv.annealed} below.

\medskip
The second key idea for the proof of Theorem~\ref{thm-mainweak} is to consider the largest eigenvalue $\lambda_1(\mathbf{X}^{(N)})$ and the corresponding eigenvector $\mathbf{v}_1(\mathbf{X}^{(N)})$ jointly and fixing a profile $\rho(\mathbf{v}_1(\mathbf{X}^{(N)}))$. This step was first introduced in~\cite{DGH2024}, where $\rho$ was given by a vector determined from the size of the blocks of the model's discrete variance profile (resp. the blocks of a discrete approximation of the continuous profile). Note that for random matrices with a discrete variance profile, the MDE~\eqref{eq-MDE} simplifies to a vector equation, whereas the model~\eqref{eq-model} is governed by a $L\times L$ matrix equation~\eqref{eq-tensorMDE}. It thus seems natural to expect that $\rho$ should take $\mathbf{v}_1(\mathbf{X}^{(N)})$ to a suitable $L\times L$ matrix in our setting.

\medskip
To define the profile, we consider vectors $\mathbf{w}\in\R^{NL}$ as being built from $L$ distinct smaller blocks $w_1,\dots,w_L\in\R^N$ such that $\mathbf{w}=(w_1,\dots,w_L)$. Next, we introduce the family of (orthogonal) projections
\begin{equation}\label{eq-defPi}
\Pi_j:\R^{NL}\rightarrow\R^{NL},\quad \mathbf{w}\mapsto (0,\dots,0,w_j,0,\dots,0)
\end{equation}
for $j=1,\dots,L$ that projects a vector $\mathbf{w}$ onto its $j$-th block. Further, for $i,j\in[L]$ let $\mathbf{P}^{i\leftrightarrow j}$ denote the permutation matrix that exchanges the $i$-th and $j$-th block of $\mathbf{w}$, e.g.,
\begin{equation}\label{eq-defP}
\mathbf{P}^{i\leftrightarrow j}\mathbf{w}=(w_1,\dots,w_{i-1},w_j,w_{i+1},\dots,w_{j-1},w_i,w_{j+1},\dots,w_L)
\end{equation}
if $1<i<j<L$. In this notation, we define $\rho:\R^{NL}\times\R^{NL}\rightarrow\R^{L\times L}$ through
\begin{equation}\label{eq-defrho}
\rho(\mathbf{w}_1,\mathbf{w}_2)_{ij}:=\langle \mathbf{P}^{i\leftrightarrow j}\Pi_i\mathbf{w}_1,\Pi_j\mathbf{w}_2\rangle+\langle \mathbf{P}^{i\leftrightarrow j}\Pi_i\mathbf{w}_2,\Pi_j\mathbf{w}_1\rangle,
\end{equation}
and abbreviate $\rho(\mathbf{w}):=\rho(\mathbf{w},\mathbf{w})/2$. Observe that the renormalized Wishart matrix in~\eqref{eq-defYtilde} can also be written as $\widetilde{Y}=\rho(\mathbf{u})$, where $\mathbf{u}$ is sampled uniformly from $\mathds{S}^{NL-1}$. Note that $\rho(\mathbf{u})$ belongs to $\mathrm{Sym}_L^{+,1}(\R)$.

\medskip
For $\Psi\in\mathrm{Sym}_L^{+,1}(\R)$, $x>r_\infty$, and $\delta,\kappa>0$, define now
\begin{equation}\label{eq-defPNdeltaPsi}
P^{(N)}_{\delta,\kappa}(x,\Psi):= \frac{1}{N}\ln\big(\P[\{|\lambda_1(\mathbf{X}^{(N)})-x|\leq\delta\}\cap\{\|\rho(\mathbf{v}_1(\mathbf{X}^{(N)}))-\Psi\|\leq\kappa\}]\big).
\end{equation}
{
We shall first prove the large deviation upper bound
\begin{proposition}\label{prop-weakLDPsup}
For every $x>r_\infty$ and $\Psi\in\mathrm{Sym}_L^{+,1}(\R)$, we have
\begin{displaymath}
\lim_{\delta,\kappa\rightarrow0}\limsup_{N\rightarrow\infty} P_{\delta,\kappa}^{(N)}(x)\leq\inf_{\theta\geq0} -\cF(\theta,x,\Psi).
\end{displaymath}
\end{proposition}

We recall that $\Psi\in\mathrm{Sym}_L^{+,1}(\R)$ is symmetric and positive semi-definite by definition, and satisfies $\Tr(\Psi)=1$. Together, these conditions imply that $\|{\Psi}\|\leq1$ and it is readily checked that $\mathrm{Sym}_L^{+,1}(\R)$ is, in fact, a compact set. {This allows to cover this set, for every $\varepsilon>0$,  by a finite union of open balls $B(\Psi_i,\delta_i(\varepsilon))$ centered at $\Psi_i$ and with radius $\delta_i(\varepsilon)>0$ so that 
$$\lim_{\kappa \downarrow 0} \limsup_{N\rightarrow \infty}\frac{1}{N}\ln P^{(N)}_{\delta_i(\varepsilon),\kappa}(x,\Psi_i)\le -\sup_{\theta\ge 0} \mathcal F(\theta,x,\Psi_i)+\varepsilon.$$
As a consequence, we deduce from Proposition~\ref{prop-weakLDPsup} that 
\begin{equation}\label{coveri}
\lim_{\delta \downarrow 0} \limsup_{N\rightarrow \infty}\frac{1}{N}\ln P^{(N)}_{\delta,\kappa}(x,\Psi)\le
\max_i \{-\sup_{\theta\ge 0} \mathcal F(\theta,x,\Psi_i)\}+\varepsilon,
\end{equation}
and the weak  large deviation upper bound from Theorem \ref{thm-mainweak} follows by taking $\varepsilon$ going to zero. Thus, estimating $\smash{P^{(N)}_\delta(x)}$ in Theorem~\ref{thm-mainweak} reduces to estimating $P^{(N)}_{\delta,\kappa}(x,\Psi)$ for fixed~$\Psi$.}}
 Over the course of the proof, we will need to restrict ourselves to profiles $\Psi$ such that $\Tr [ \Psi^T S(\Psi)] \neq 0$.  For this we introduce the ``approximate'' version of $\rate$ {given, for~${\eps >0}$, by}
\begin{equation}\label{eq-defapproxI}
\rate(x, \eps) = \inf_{\substack{\Psi\in\mathrm{Sym}_L^{{+,1}}(\R),\\ \Tr[\Psi^T\cS(\Psi)]\geq \eps}}\sup_{\theta\geq0}\cF(\theta,x,\Psi).
\end{equation} 

We will in fact prove our large deviation lower and upper bound with $ \rate(x, \eps)  $ for any $\eps >0$ small enough. This will prove then that for such small $\eps>0$, $\rate(x,\eps)$ does not depend on $\eps>0$ and is equal to $\rate(x)$.

\medskip
With these tools in hand we can start the proof of Theorem~\ref{thm-mainweak}, starting with the convergence of the annealed spherical integral.

\section{Proof of Proposition~\ref{prop-conv.annealed} (Convergence of the Annealed Spherical Integral)}\label{sect-conv.annealed}

By definition of the model, we have that
\begin{align*}
\E_W\big[\re^{\theta NL\langle \mathbf{u},\mathbf{X}^{(N)}\mathbf{u}\rangle}\big]=\re^{\theta NL\langle\mathbf{u},(A_0\otimes \Id_N)\mathbf{u}\rangle}\prod_{j=1}^k\E_W\big[\re^{\theta NL\langle \mathbf{u},(A_j\otimes W_j)\mathbf{u}\rangle}\big],
\end{align*}
and can thus focus on the summands of $X$ separately. Rewriting the vector $u\in\mathds{S}^{NL-1}$ as $\mathbf{u}=(u_1,\dots,u_L)$ with $u_1,\dots,u_L\in\R^{N}$, we obtain
\begin{align*}
\langle \mathbf{u},(A_j\otimes W_j)\mathbf{u}\rangle&=\sum_{c,d=1}^L(A_j)_{cd}\langle u_c,W_ju_d\rangle=\sum_{a,b=1}^N(W_j)_{ab}\sum_{c,d=1}^L(A_j)_{cd}(u_c)_a(u_d)_b.
\end{align*}
Recall that $(W_j)_{ab}$ are Gaussian random variables and that the elements in the upper triangle of $W$ are independent. Denoting
\begin{equation}\label{eq-defL}
\mathscr{L}_{a,b}(t)=\ln\big(\E_W\big[\re^{t(W_j)_{ab}}\big]\big)=\frac{1+\I_{a=b}}{2N}t^2
\end{equation}
for $1\leq a\leq b\leq N$, we compute
\begin{align}
\prod_{j=1}^k\E_W\big[\re^{\theta NL\langle \mathbf{u},(A_j\otimes W_j)\mathbf{u}\rangle}\big]&=\prod_{j=1}^k\exp\Bigg(\sum_{a=1}^N\mathscr{L}_{a,a}\Big(NL\theta\sum_{c,d}(A_j)_{cd}(u_c)_a(u_d)_a)\Big)\NN\\
&\quad\quad\quad\quad+\sum_{1\leq a<b\leq N}\mathscr{L}_{a,b}\Big(2NL\theta\sum_{c,d}(A_j)_{cd}(u_c)_a(u_d)_b\Big)\Bigg)\label{eq-innerprod}\\
&=\exp\Big(NL^2\theta^2\sum_{j=1}^k\Tr[\rho(\mathbf{u})A_j\rho(\mathbf{u})A_j]\Big),\NN
\end{align}
where we used that $\rho(\mathbf{u})_{i,j}=\langle u_i,u_j\rangle$ in the last step. As further
\begin{displaymath}
\re^{\theta NL\langle \mathbf{u},(A_0\otimes\Id_N)\mathbf{u}\rangle}
=\exp\Big(\theta NL\Tr[A_0^T\rho(\mathbf{u})]\Big),
\end{displaymath}
we conclude that
\begin{equation}\label{eq-inter}
\E_W\big[\re^{\theta NL\langle \mathbf{u},\mathbf{X}^{(N)}\mathbf{u}\rangle}\big]=e^{N\mathscr{F}_\theta(\rho(\mathbf{u}))},
\end{equation}
where the function $\mathscr{F}_\theta$ on the right-hand side of~\eqref{eq-inter} is given by
\begin{displaymath}
\mathscr{F}_\theta:(x_{11},\dots,x_{LL})\mapsto L^2\theta^2\sum_{j=1}^k\sum_{c,d=1}^L\sum_{c',d'=1}^L(A_j)_{cd}(A_j)_{c'd'}x_{cc'}x_{dd'}+L\theta\sum_{c,d=1}^L(A_0)_{cd}x_{cd}.\end{displaymath}
Observe that $\mathscr{F}_\theta$ is bounded and continuous on $\mathrm{Sym}_L^{+,1}(\R)$.
Therefore, we find $
\varepsilon(\delta)$ going to zero with $\delta$ so that 
\begin{equation}\label{eq-contr}\left| \ln \frac{\E_{\mathbf{u}}\E_W[
e^{\theta N L\langle \mathbf{u}, \mathbf{X}^{(N)}\mathbf{u}\rangle}1_{\|\rho(\mathbf{u}) -\Psi\|\le \delta}
] }{e^{N\mathscr{F}_\theta(\Psi) }\mathbb P(\|\rho(\mathbf{u}) -\Psi\|\le \delta)}\right|\le \varepsilon(\delta)N.
\end{equation}
To complete our proof it is therefore enough to show that $\rho(\mathbf{u})$ satisfies a LDP.
\begin{lemma}\label{lem-WishartLDP}
In the above setup, the law of the (upper triangular) entries of $\rho(\mathbf{u})$ satisfies an LDP with speed $N$ and good rate function
\begin{displaymath}
\widetilde{\rate}:\mathrm{Sym}_L^{+,1}(\R)\rightarrow [0,\infty],\quad  \widetilde{\rate}(X)=-\frac{1}{2}(\ln(\det(X))-L\ln(L)).
\end{displaymath}
\end{lemma}

The proof of Lemma~\ref{lem-WishartLDP} follows standard arguments (cf., e.g., ~\cite[Lem.~4.2]{Husson2022}), but requires a careful treatment of the determinant. We include it in Appendix~\ref{app-WishartLDP}. Given Lemma~\ref{lem-WishartLDP} and~\eqref{eq-contr}, Proposition~\ref{prop-conv.annealed} readily follows.

\section{Proof of Theorem~\ref{thm-main}: Weak LDP Upper Bound}\label{sect-wLDPsup}
The goal of this section is to prove the upper bound of the weak LDP in Theorem~\ref{thm-mainweak}.
The key tool to establishing Proposition~\ref{prop-weakLDPsup} is a bound for~\eqref{eq-defPNdeltaPsi} in terms of~\eqref{eq-Fubinitrick} that allows harnessing the convergence of the annealed spherical integral. Before we formulate it, we note the following property of the rate function.
\begin{lemma}\label{lem-optimizationreg}
For any profile $\Psi\in\mathrm{Sym}_L^{+,1}(\R)$, any $x>r_\infty$, and $\theta$ such that $(-m_{\mu_\infty})^{-1}(2\theta)>x$, we have that
\begin{displaymath}
K(\theta, \phi(\theta, x, \Psi))=LJ_{\mu_\infty}(x,\theta),
\end{displaymath}
i.e., $\cF(\theta,x,\Psi)=0$ for $\theta<-m_{\mu_\infty}(x)/2$. 
\end{lemma}
Recall that we are aiming for a LDP with speed $N$ while $\mathbf{X}^{(N)}$ has size $NL\times NL$. As a consequence, we obtain an additional factor $L$ over~\eqref{eq-defJ} in the rate function.

\begin{proof}[Proof of Lemma~\ref{lem-optimizationreg}]
First, observe that for the range of $\theta$ considered, we have that
\begin{align*}
\phi(\theta,x,\Psi)&=\varphi(\theta,x)=-\frac{M((-m_{\mu_\infty})^{-1}(2\theta))}{2\theta L},\\ J_{\mu_\infty}(x,\theta)&=\theta (-m_{\mu_\infty})^{-1}(2\theta)-\frac{1+\ln(2\theta)}{2}-\frac{1}{2}\int_\R\ln\big((-m_{\mu_\infty})^{-1}(2\theta)-y\big)\dx \mu_\infty(y).
\end{align*}
Next, we introduce the variable $t:=(-m_{\mu_\infty})^{-1}(2\theta)$. This allows rewriting
\begin{align*}
K(\theta,\phi(\theta,x,\Psi))&=L^2\theta^2\Tr[\phi(\theta,x,\Psi)^T\cS[\phi(\theta,x,\Psi)]]+L\theta\Tr[A_0^T\phi(\theta,x,\Psi)]\\
&\quad+\frac{1}{2}\ln(\det[\phi(\theta,x,\Psi)])-L\ln(L)\\
&=\frac{1}{4}\Tr[M(t)^T\cS[M(t)]]-\frac{1}{2}\Tr[A_0^TM(t)]-\frac{L}{2}\ln(-Lm_{\mu_\infty}(t))\\
&\quad+\frac{1}{2}\ln(\det[-M(t)])-L\ln(L),
\end{align*}
as well as
\begin{align*}
J_{\mu_\infty}(x,\theta)&=\frac{1}{2}\Big(-tm_{\mu_\infty}(t)-1-\ln(-m_{\mu_\infty}(t))-\int\ln(t-y)\dx \mu_\infty\Big).
\end{align*}
Lastly, we compute
\begin{equation}\label{eq-tderiv}
\frac{\dx}{\dx t}\Big(K(\theta,\phi(\theta,x,\Psi))-LJ_{\mu_\infty}(x,\theta)\Big)=\frac{1}{2}\Tr\Big[M'(t)\Big(\cS[M(t)-A_0+M(t)^{-1}+t\Id_L]\Big)\Big],
\end{equation}
where $M'$ denotes the (entry-wise) derivative w.r.t. to the argument. Recalling that $M(t)$ satisfies the MDE~\eqref{eq-tensorMDE}, the term in the round brackets, and hence the derivative in~\eqref{eq-tderiv}, vanishes. Observing that $\lim_{t\rightarrow\infty}\theta(t)=\lim_{t\rightarrow\infty}-m_{\mu_\infty}(t)/2=0$ and noting that also
\begin{displaymath}
\lim_{t\rightarrow\infty}\Big(K(\theta,\phi(\theta,x,\Psi))-LJ_{\mu_\infty}(x,\theta)\Big)=0,
\end{displaymath}
we conclude that $K(\theta,\phi(\theta(t),x,\Psi))-LJ_{\mu_\infty}(x,\theta)=0$ for $t>x$, which yields the claim.
\end{proof}

Using Lemma~\ref{lem-optimizationreg}, proving Proposition~\ref{prop-weakLDPsup} reduces to showing
\begin{equation}\label{eq-upperboundgoal}
\lim_{\delta,\kappa\rightarrow0}\limsup_{N\rightarrow\infty}P^{(N)}_{\delta,\kappa}(x,\Psi)\leq-\cF(\theta,x,\Psi)
\end{equation}
only for $\theta>(-m_{\mu_\infty})(x)/2$ rather than all $\theta>0$.

\medskip
The rest of the section focuses on the proof of~\eqref{eq-upperboundgoal}. Firstly, we may assume without loss of generality that
\begin{equation}\label{finite}
\lim_{\delta,\kappa\rightarrow0}\limsup_{N\rightarrow\infty}P^N_{\delta,\kappa}(x,\Psi)>-\infty.
\end{equation}
The desired result is then immediate from the following bound.
{
\begin{proposition}\label{prop-sphintloc}
There is a function $o_{\delta,\kappa}(1)$ with $\lim_{\delta,\kappa\rightarrow0}o_{\delta,\kappa}(1)=0$ such that for every $\delta,\kappa>0$, $x\geq r_\infty$, $\theta>(-m_{\mu_\infty})(x)/2$, and profile $\Psi\in\mathrm{Sym}_{L}^{+,1}(\R)$ it holds that
\begin{displaymath}
P_{\delta,\kappa}^{(N)}(x,\Psi)\leq\frac{1}{N}\E_W\E_{\mathbf{u}}[\I_{\{\|\rho(\mathbf{u})-\Phi(\theta,x,\Psi)\|<o_{\delta,\kappa}(1)\}}\re^{\theta NL\langle\mathbf{u},\mathbf{X}^{(N)}\mathbf{u}\rangle}]-LJ_\mu(x,\theta)+\rem(N,\delta,\kappa),
\end{displaymath}
where the error term satisfies $\lim_{\delta,\kappa\rightarrow0}\lim_{N\rightarrow\infty}\rem(N,\delta,\kappa)=0$.
\end{proposition}}

Combining Proposition~\ref{prop-sphintloc} and Proposition \ref{prop-conv.annealed} readily yields \eqref{eq-upperboundgoal}, which completes the poof of Proposition~\ref{prop-weakLDPsup}. It thus remains to show Proposition~\ref{prop-sphintloc}.

\subsection{Proof of Proposition~\ref{prop-sphintloc} (Localization of the Spherical Integral)}
Before we give the proof of Proposition~\ref{prop-sphintloc}, we collect the necessary tools and notation. At the core of the argument are two key ingredients:
\begin{enumerate}
\item[(1.)] a concentration bound for the empirical spectral measure that resolves the convergence for each component {$\mu_{i,j}$ } in~\eqref{eq-Stieltjestransf} individually,
\item[(2.)] a set $S_\eps\subset\mathds{S}^{NL-1}$ for which the auxiliary quantities introduced during the tilting argument are small and $\E_{\mathbf{u}}[\I_{S_\eps}\exp(\theta NL\langle\mathbf{u},\mathbf{X}^{(N)}\mathbf{u}\rangle)](I_{NL}(\mathbf{X}^{(N)},\theta))^{-1}$ is bounded from below.
\end{enumerate}
Together with the convergence of the spherical integral in Lemma~\ref{lem-convsphint}, the bounds in (1.) and (2.) allow us to control all error terms that arise in the proof of Proposition~\ref{prop-sphintloc}. We start by introducing the notation needed to rigorously state (1.).

\medskip
Recall that $\Pi_j$ denotes the projection of $\mathbf{w}\in\R^{NL}$ onto its $j$-th block of length $N$ as introduced in~\eqref{eq-defPi}, and $\mathbf{P}^{i\leftrightarrow j}$ denotes the permutation matrix interchanging the $i$-th and $j$-th block of $\mathbf{w}$ as defined in~\eqref{eq-defP}. Given a symmetric matrix $\mathbf{H}\in\R^{NL\times NL}$, let
\begin{equation}\label{eq-defmuproj}
\mu_{\mathbf{H}}(\Pi_i,\Pi_j)=\frac{1}{N}\sum_{a=1}^{NL}\langle \mathbf{P}^{i\leftrightarrow j}\Pi_i\mathbf{v}_a(\mathbf{H}),\Pi_j\mathbf{v}_a(\mathbf{H})\rangle \delta_{\lambda_a(\mathbf{H})},
\end{equation}
where $\mathbf{v}_a(\mathbf{H})$ and $\lambda_a(\mathbf{H})$ denote the eigenvectors and eigenvalues of $\mathbf{H}$, respectively, and we used boldface $\mathbf{v}_a$ to indicate that the eigenvectors have length $NL$. Note that the definition in~\eqref{eq-defmuproj} gives back the usual empirical spectral measure $\mu_\mathbf{H}$ for~${L=1}$, as the eigenvectors only consist of one block of length $N$. However, for general~$L$, the coefficients of the $\delta_{\lambda_{i}(\mathbf{H})}$ do not need to be non-negative numbers. But, observing that the corresponding total variation measure is given by
\begin{displaymath}
|\mu_{\mathbf{H}}(\Pi_i,\Pi_j)|=\frac{1}{N}\sum_{a=1}^{NL}|\langle \mathbf{P}^{i\leftrightarrow j}\Pi_i\mathbf{v}_a(\mathbf{H}),\Pi_j\mathbf{v}_a(\mathbf{H})\rangle| \delta_{\lambda_a(\mathbf{H})},
\end{displaymath}
and that in particular
\begin{equation}\label{eq-complexRadon}
|\mu_{\mathbf{H}}(\Pi_i,\Pi_j)|(\R)=\frac{1}{N}\sum_{a=1}^{NL}|\langle \mathbf{P}^{i\leftrightarrow j}\Pi_i\mathbf{v}_a(\mathbf{H}),\Pi_j\mathbf{v}_a(\mathbf{H})\rangle|\leq L,
\end{equation}
it follows that $\mu_{\mathbf{H}}(\Pi_i,\Pi_j)$ is well-defined as a signed Radon measure. Note that we may also interpret $(\mu_{\mathbf{H}}(\Pi_i,\Pi_j))_{i,j=1}^L$ as a measure taking values in $\mathrm{Sym}_L^+(\R)$.

\medskip
For a continuous bounded function $f$, we note the identity
\begin{align*}
\mu_{\mathbf{H}}(\Pi_j,\Pi_j)[f]&=\int_{\R}f\dx \mu_{\mathbf{H}}(\Pi_i,\Pi_j)=\frac{1}{N}\sum_{a=1}^{NL}\langle \mathbf{P}^{i\leftrightarrow j}\Pi_i\mathbf{v}_a(\mathbf{H}),\Pi_j\mathbf{v}_a(\mathbf{H})\rangle f(\lambda_a)\\
&=\frac{1}{N}\Tr[\mathbf{P}^{i\leftrightarrow j}\Pi_if(\mathbf{H})\Pi_j]\\
&=(E_{ij}\otimes \frac{1}{N}\Tr[\cdot])[f(\mathbf{H})].
\end{align*}

With these definitions in place, we continue by giving a concentration result for $\mu_{\mathbf{X}^{(N)}}(\Pi_i,\Pi_j)$. To quantify the distance between two (possibly signed) measures $\nu_1,\nu_2$, we use
\begin{equation}\label{eq-defKR}
d_{KR}(\nu_1,\nu_2)=\sup_{\substack{\|f\|_\infty\leq1,\\ \mathrm{Lip}(f)\leq1}}\int f\dx(\nu_1-\nu_2),
\end{equation}
where $\mathrm{Lip}(f)$ denotes the Lipschitz constant associated to a Lipschitz function $f$. The metric~\eqref{eq-defKR} is usually referred to as \emph{Kantorovich-Rubinstein distance} in the literature. We remark that $d_{KR}$ may not induce the same topology as weak convergence on the space of signed measures, however, the induced topologies coincide for certain subsets. We note the following application of~\cite[Thm.~2.1]{AfoninBogachev2023} to our setting.

\begin{lemma}
Consider $(\R,|\cdot|)$ where $|\cdot|$ denotes the usual absolute value. Then $d_{KR}$ generates the weak topology on every uniformly bounded in total variation and uniformly tight set of signed measures on $\R$.
\end{lemma}

{We remark that the additional conditions are indeed satisfied in our setting, as~\eqref{eq-complexRadon} provides a uniform bound on the total variation measure and the support of the $\mu_{\mathbf{X}^{(N)}}(\Pi_j,\Pi_j)$ is well-controlled by a bound on $\|\mathbf{X}^{(N)}\|$ (cf. Lemma~\ref{lem-norm_exptight} below)}. The desired concentration result can now be stated as follows. 

\begin{theorem}\label{thm-muproj.conc}
For every $\eps>0$ we have
\begin{displaymath}
\limsup_{N\rightarrow\infty}\frac{1}{N}\ln[\P(\exists i,j\in[L]: d_{KR}(\mu_{\mathbf{X}^{(N)}}(\Pi_i,\Pi_j),\mu_{i,j})\geq\eps)]=-\infty,
\end{displaymath}
where the measure $\mu_{i,j}$ was introduced in~\eqref{eq-Stieltjestransf}.
\end{theorem}
The proof of this theorem is given in Section~\ref{sect-concentration}, after finishing the proof of Proposition~\ref{prop-sphintloc}. Note that Theorem~\ref{thm-muproj.conc} is stronger than the (exponential) concentration for the empirical eigenvalue measure, as we aim to control $(\mu_{\mathbf{X}^{(N)}}(\Pi_i,\Pi_j))_{i,j=1}^L$ as a matrix and not just its normalized trace.

\medskip
Next, we note that the norms of~\eqref{eq-model} constitute an exponentially tight sequence.
\begin{lemma}\label{lem-norm_exptight}
The sequence $(\|\mathbf{X}^{(N)}\|)_N$ is exponentially tight. More precisely, there is a constant $C>0$ such that every $\cC>0$ it holds that
\begin{displaymath}
\frac{1}{N}\ln\big(\P(\|\mathbf{X}^{(N)}\|\geq\cC)\big)\leq C-{\frac{L\cC}{4}}.
\end{displaymath}
\end{lemma}

\begin{proof}[Proof of Lemma~\ref{lem-norm_exptight}]
The claim follows from a standard net argument, see, e.g.,~\cite[Lem.~1.8]{GuionnetHusson2020}. We give the key steps for the convenience of the reader. Let $\cR_N$ denote an $1/2$-net of $\mathds{S}^{NL-1}$, i.e., for each vector $\mathbf{w}\in\mathds{S}^{NL-1}$ there is a $\mathbf{r}_{\mathds{v}}$ such that $\|\mathbf{w}-\mathbf{r}_{\mathds{v}}\|_2\leq1/2$. We note that we may choose a net with $|\cR_N|\leq 3^{NL}$ and that by elementary estimations
\begin{displaymath}
\|\mathbf{X}^{(N)}\|\leq 2\sup_{\mathbf{r}\in \cR_N}\|\mathbf{X}^{(N)}\mathbf{r}\|\leq 4\sup_{\mathbf{r}_1,\mathbf{r}_2\in \cR_N}\langle \mathbf{X}^{(N)}\mathbf{r}_1,\mathbf{r}_2\rangle.
\end{displaymath}
Moreover, for $\cC>0$, applying the Markov inequality yields
\begin{align*}
\P(\|\mathbf{X}^{(N)}\|\geq\cC)&\leq 9^{NL}\sup_{\mathbf{r}_1,\mathbf{r}_2\in \cR_N}\P(\langle \mathbf{X}^{(N)}\mathbf{r}_1,\mathbf{r}_2\rangle\geq\cC/4)\\
&\leq 9^{NL}\sup_{\mathbf{r}_1,\mathbf{r}_2\in \cR_N}\frac{\E_W[\exp(NL\langle \mathbf{X}^{(N)}\mathbf{r}_1,\mathbf{r}_2\rangle)]}{\exp(NL\cC/4)}\\
&\leq\exp\big(N(2\ln(3)L+kL^6\max_{j\in[k]}\|A_j\|^2+L^3\|A_0\|-L\cC/4)\big),
\end{align*}
where the last bound follows from~\eqref{eq-innerprod} and the analogous computation for the $A_0$ term of~\eqref{eq-model}.
\end{proof}

Combining the quantities in Theorem~\ref{thm-muproj.conc} and Lemma~\ref{lem-norm_exptight}, we introduce the following sets: For a sequence $(\eps_N)_N$ with $\eps_N>0$ and $\lim_{N\rightarrow\infty}\eps_N=0$, define
\begin{equation}\label{eq-defOmegaN}
\Omega_N:=\{\mathbf{H}\in\mathrm{Sym}_{NL}(\R): \forall i,j\in[L], d_{KR}(\mu_{\mathbf{H}}(\Pi_i,\Pi_j),\mu_{ij})\leq\eps_N, \|\mathbf{H}\|\leq \cC\}.
\end{equation}
By construction,~\eqref{eq-defOmegaN} has a large probability.

\begin{lemma}\label{lem-largeproba}
For any $K>0$ there exists $\cC>0$ and a sequence $(\eps_N)_N$ with $\eps_N>0$ and $\lim_{N\rightarrow\infty}\eps_N=0$ such that
\begin{displaymath}
\limsup_{N\rightarrow\infty}\frac{1}{N}\ln\big(\P(\mathbf{X}^{(N)}\notin\Omega_N)\big)\leq -K.
\end{displaymath}
\end{lemma}

\begin{proof}[Proof of Lemma~\ref{lem-largeproba}]
This is immediate from the above results: The existence of a suitable sequence $(\eps_N)_N$ follows from Theorem~\ref{thm-muproj.conc} and we use Lemma~\ref{lem-norm_exptight} to pick $\cC>0$.
\end{proof}

Next, consider the following bound for $\Psi\in\mathrm{Sym}_L^{+,1}(\R)$, $x>r_\infty$, and $\delta,\kappa>0$
\begin{align*}
&\P[\{\mathbf{X}^{(N)}\in\Omega_N\}\cap\{|\lambda_1(\mathbf{X}^{(N)})-x|\leq\delta\}\cap\{\|\rho(\mathbf{v}_1(\mathbf{X}^{(N)}))-\Psi\|\leq\kappa\}]\\
&\geq \exp(NP^N_{\delta,\kappa}(x,\Psi))(1-\P(\Omega_N^c)\exp(-NP^N_{\delta,\kappa}(x,\Psi))),
\end{align*}
where Lemma~\ref{lem-largeproba} assures that for $\mathcal C$ large enough, we can find $o_N(1)$ going to zero when~$N$ goes to infinity so that
\begin{displaymath}
P^N_{\delta,\kappa}(x,\Psi)\leq\frac{1+o_N(1)}{N}\ln\big(\P[\{\mathbf{X}^{(N)}\in\Omega_N\}\cap\{|\lambda_1(\mathbf{X}^{(N)})-x|\leq\delta\}\cap\{\|\rho(\mathbf{v}_1(\mathbf{X}^{(N)}))-\Psi\|\leq\kappa\}]\big),
\end{displaymath}
provided we choose the constants according to \eqref{finite}. This introduces $\Omega_N$ into the estimation of $P^{(N)}_{\delta,\kappa}(x,\Psi)$ and motivates the definition of the set $B_{\delta,\kappa, N}(x,\Psi)$ as
\begin{equation}\label{eq-defBdeltaN}
B_{\delta,\kappa, N}(x,\Psi)=\{\mathbf{X}^{(N)}\in\Omega_N\}\cap\{|\lambda_1(\mathbf{X}^{(N)})-x|\leq\delta\}\cap\{\|\rho(\mathbf{v}_1(\mathbf{X}^{(N)}))-\Psi\|\leq\kappa\}.
\end{equation}
We include it into the computation via tilting by spherical integrals. More precisely, thanks to Lemma \ref{lem-convsphint},
\begin{align}
P^N_{\delta,{\kappa}}(x,\Psi)&\leq\frac{1+o_N(1)}{N}\ln\E\big[\I_{\{\mathbf{X}^{(N)}\in B_{\delta,{\kappa}, N}(x,\Psi)\}}\frac{I_{NL}(\mathbf{X}^{(N)},\theta)}{I_{NL}(\mathbf{X}^{(N)},\theta)}\big]\label{eq-ubtilting1}\\
&\leq\frac{1}{N}\ln\E_W\big[\I_{\{\mathbf{X}^{(N)}\in B_{\delta,{\kappa}, N}(x,\Psi)\}}\E_{\mathbf{u}}[\re^{NL\theta\langle\mathbf{u},\mathbf{X}^{(N)}\mathbf{u}\rangle}]\big]-LJ_{\mu_\infty}(x,\theta)-\rem(N,\delta,{\kappa}),\NN
\end{align}
where $\lim_{\delta,\kappa\rightarrow0}\lim_{N\rightarrow\infty}\rem(N,\delta,{\kappa})=0$ due to the convergence of the spherical integral. Recall that $\mu_\infty$ denotes the limiting spectral measure for the model.

\medskip
Lastly, we define the set $S_\eps$ for $\lambda\in\mathbb R$ and a vector $\mathbf{w}\in\R^{NL}$ by
\begin{equation}\label{eq-defSeps}
S_\eps(\lambda,\mathbf{w}):=\{\mathbf{u}\in\mathds{S}^{NL-1}:\max\{\|\rho(\Pi^{\perp \mathbf{w}_1}\mathbf{u})-\varphi(\lambda,\theta)\|_\infty,\|\rho(\Pi^{\perp \mathbf{w}}\mathbf{u},\mathbf{w})\|_{\infty}\}<\eps\},
\end{equation}
where $\Pi^{\perp \mathbf{w}}$ denotes the projection to the orthocomplement of $\mathbf{w}$ and $\varphi$ was defined in~\eqref{eq-defvphi}. This allows formulating the second main tool for this section.

\begin{proposition}\label{prop-tiltedonsphere}
Fix $\delta,\eps,\kappa>0$. For any $\mathbf{H}\in B_{\delta,\kappa, N}(x,\Psi)$ such that $\lambda_1(\mathbf{H})-3\delta>r_\infty$ and $m_{\mu_\infty}(\lambda_1(\mathbf{H})-3\delta)>-2\theta$, we have
\begin{displaymath}
\frac{\E_{\mathbf{u}}[\mathds{1}_{S_\eps(\lambda_1(\mathbf{H}),\mathbf{v}_1(\mathbf{H}))}\re^{\theta NL\langle \mathbf{u},\mathbf{Hu}\rangle}]}{\E_{\mathbf{u}}[\re^{\theta NL \langle\mathbf{u},\mathbf{Hu}\rangle}]}\geq \re^{-o(N)}
\end{displaymath}
where $S_\eps$ was defined in~\eqref{eq-defSeps} and the error satisfies $\lim_{N\rightarrow\infty}o(N)/N=0$.
\end{proposition}

We prove Proposition~\ref{prop-tiltedonsphere} in Section~\ref{sect-tiltedonsphere}. Assuming all of the above tools given, we assemble them to prove Proposition~\ref{prop-sphintloc}. 

\begin{proof}[Proof of Proposition~\ref{prop-sphintloc}]
After the initial step of tilting the measure by spherical integrals in~\eqref{eq-ubtilting1}, Proposition~\ref{prop-tiltedonsphere} allows to replace $\E_{\mathbf{u}}[\re^{\theta NL \langle\mathbf{u},\mathbf{X}^{(N)}\mathbf{u}\rangle}]$ by a term involving the set~\eqref{eq-defSeps}. More precisely, we have
\begin{align}
P^N_{\delta,\kappa}(x,\Psi)&\leq\frac{1}{N}\E_W\big[\I_{\{\mathbf{X}^{(N)}\in B_{\delta,{\kappa}, N}(x,\Psi)\}}\E_{\mathbf{u}}[\mathds{1}_{S_\eps(\mathbf{v}_1(\mathbf{X}^{(N)}),\lambda_1(\mathbf{X}^{(N)}))}\re^{\theta NL\langle \mathbf{u},\mathbf{X}^{(N)}\mathbf{u}\rangle}]\big]\NN\\
&\quad-LJ_{\mu_\infty}(x,\theta)-\rem(N,\delta,{\kappa}),\label{eq-ubtilting2}
\end{align}
where the error satisfies $\lim_{\delta,{\kappa}\rightarrow0}\lim_{N\rightarrow\infty}\rem(N,\delta,\kappa)=0$. Note that the error term $\rem(N,\delta,\kappa)$ in~\eqref{eq-ubtilting2} may be different from that in~\eqref{eq-ubtilting1} although the limiting behavior is the same.

\medskip
Next, we focus on the conditions in~\eqref{eq-defSeps}. Decomposing
\begin{displaymath}
\mathbf{u}=\langle \mathbf{v}_1(\mathbf{X}^{(N)}),\mathbf{u}\rangle\mathbf{v}_1(\mathbf{X}^{(N)})+\Pi^{\perp \mathbf{v}_1(\mathbf{X}^{(N)})}\mathbf{u},
\end{displaymath}
it follows that
\begin{align}
\rho(\mathbf{u})_{ij}&=\langle \mathbf{P}^{i\leftrightarrow j}\Pi_i\mathbf{u},\Pi_j\mathbf{u}\rangle\NN\\
&=\langle\mathbf{v}_1(\mathbf{X}^{(N)}),\mathbf{u}\rangle^2\rho(\mathbf{v}_1(\mathbf{X}^{(N)}))_{ij}+\langle\mathbf{v}_1(\mathbf{X}^{(N)}),\mathbf{u}\rangle\rho(\mathbf{v}_1(\mathbf{X}^{(N)}),\Pi^{\perp \mathbf{v}_1(\mathbf{X}^{(N)})}\mathbf{u})_{ij}\NN\\
&\quad +\rho(\Pi^{\perp \mathbf{v}_1(\mathbf{X}^{(N)})}\mathbf{u})_{ij}.\label{eq-rhoudecomp}
\end{align}
Consider now $\mathbf{u}$ in $S_\eps(\mathbf{v}_1(\mathbf{X}^{(N)}),\lambda_1(\mathbf{X}^{(N)}))$ and $\mathbf{X}^{(N)}$ so that $\|\rho(\mathbf{v}_1(\mathbf{X}^{(N)}))-\Psi\|\leq\kappa$. Then, we have
\begin{align*}
\langle\mathbf{v}_1(\mathbf{X}^{(N)}),\mathbf{u}\rangle^2&=1-\|\Pi^{\perp \mathbf{v}_1(\mathbf{X}^{(N)})}\mathbf{u}\|^2=1-\sum_{j=1}^L\langle\Pi_j\Pi^{\perp \mathbf{v}_1(\mathbf{X}^{(N)})}\mathbf{u},\Pi_j\Pi^{\perp \mathbf{v}_1(\mathbf{X}^{(N)})}\mathbf{u}\rangle\\
&=1-\Tr[\rho(\Pi^{\perp \mathbf{v}_1(\mathbf{X}^{(N)})}\mathbf{u})]= 1-\Tr[\varphi(\lambda_1(\mathbf{X}^{(N)}),\theta)]+\cO(\eps).
\end{align*}
Moreover, the second term in~\eqref{eq-rhoudecomp} is $\cO(\eps)$ while the third is equal to $\varphi(\lambda_1(\mathbf{X}^{(N)}),\theta)$ up to an $\cO(\eps)$ error. It follows that
\begin{displaymath}
\rho(\mathbf{u})=\Big(1-\Tr[\varphi(\lambda_1(\mathbf{X}^{(N)}),\theta)]\Big)\Psi+\varphi(\lambda_1{(\mathbf{X}^{(N)})},\theta)+
o(\delta,{\kappa},\eps)
\end{displaymath}
where the bound $o_{\delta,\kappa,\eps}(1)$ satisfies $\lim_{\delta,\kappa, \eps\rightarrow0}o(\delta,\kappa, \eps)=0$.
Lastly, recalling the choice of $\phi(\theta,x,\Psi)$ from~\eqref{eq-defphi} and that the map $\lambda\mapsto m_{\mu_\infty}(\lambda)$ is smooth for $\lambda\in\R$ outside of the support of the limiting spectral density, we can estimate the indicators in~\eqref{eq-ubtilting2} as
\begin{align*}
&\I_{\{\mathbf{X}^{(N)}\in B_{\delta,\kappa,N}(x,\Psi)\}}\I_{{\mathbf{u}\in} S_\eps(\mathbf{v}_1(\mathbf{X}^{(N)}),\lambda_1(\mathbf{X}^{(N)}))}\\
&\leq \I_{{\mathbf{u}\in} S_\eps(\mathbf{v}_1(\mathbf{X}^{(N)}),\lambda_1(\mathbf{X}^{(N)}))}\I_{\{\|\rho(\mathbf{v}_1(\mathbf{X}^{(N)}))-\Psi\|\leq\kappa\}}\I_{\{|\lambda_1(\mathbf{X}^{(N)})-x|\leq\delta\}}\\
&\leq \I_{\{\|\rho(\mathbf{u})-\phi(\theta,x,\Psi)\|\leq o_{\delta,\kappa,\eps}(1)\}}.
\end{align*}
 Fixing a sequence $(\eps'_N)_N$ that ensures Proposition~\ref{prop-tiltedonsphere} remains true reduces 
 $o_{\delta,\kappa, \eps}(1)$ to the desired $o_\delta(1)$ error. Hence, we deduce from~\eqref{eq-ubtilting2} that
\begin{displaymath}
P^N_{\delta,\kappa}(x,\Psi)\leq\frac{1}{N}\E_W\E_{\mathbf{u}}[\I_{\{\|\rho(\mathbf{u})-\phi(\theta,x,\Psi)\|<o_{\delta,{\kappa}}(1))\}}\re^{\theta NL\langle\mathbf{u},\mathbf{X}^{(N)}\mathbf{u}\rangle}]-J_\mu(x,\theta)+\rem(N,\delta,{\kappa})
\end{displaymath}
with $\lim_{\delta\rightarrow0}\lim_{N\rightarrow\infty}\rem(N,\delta,{\kappa})=0$ and $o_{\delta,{\kappa}}(1)$ goes to zero as $\delta,\kappa$ go to zero, as claimed.
\end{proof}

It remains to establish Theorem~\ref{thm-muproj.conc} and Proposition~\ref{prop-tiltedonsphere} which were the main tools used in the proof of Proposition~\ref{prop-sphintloc}.

\subsection{Proof of Theorem~\ref{thm-muproj.conc}}\label{sect-concentration}
The proof is divided into two steps. First, we show the convergence of the expectation and then we conclude using a concentration of measure result. 

\medskip
\underline{\smash{Convergence of the expectation:}} We show that $(\E[\mu_{\mathbf{X}^{(N)}}(\Pi_i,\Pi_j)])_N$ converges vaguely to~$\mu_{ij}$ for every choice of $i,j\in[L]$. Using the tightness in Lemma~\ref{lem-norm_exptight}, this statement is then readily strengthened to the desired weak convergence.

\medskip
Let $z\in\C$ with $\Im(z)>0$. We start by computing
\begin{equation}\label{eq-defRX}
\mu_{\mathbf{X}^{(N)}}(\Pi_i,\Pi_j)[(\cdot-z)^{-1}]=(E_{ij}\otimes\frac{1}{N}\Tr[\cdot])[(\mathbf{X}^{(N)}-z\Id_{NL})^{-1}],
\end{equation}
which shows that the Stieltjes transform of $\mu_{\mathbf{X}^{(N)}}(\Pi_i,\Pi_j)$ is given by the normalized trace of the $N\times N$-block at position $(i,j)$ of the resolvent of $\mathbf{X}^{(N)}$. Let $R_{\mathbf{X}^{(N)}}(z)$ denote the matrix for which $(R_{\mathbf{X}^{(N)}}(z))_{ij}$ is given by the expectation of~\eqref{eq-defRX} and observe that $(R_{\mathbf{X}^{(N)}})_{ij}$ is, in fact, the Stieltjes transform of $\E[\mu_{\mathbf{X}^{(N)}}(\Pi_i,\Pi_j)]$. Applying~\cite[Thm.~4.3]{HST2006}, it follows that $R_{\mathbf{X}^{(N)}}(z)$ is an approximate solution to the MDE~\eqref{eq-tensorMDE}. More precisely,
\begin{displaymath}
\|\Id_L+(A_0-zI_L)R_{\mathbf{X}^{(N)}}+\sum_{j=1}^kA_jR_{\mathbf{X}^{(N)}}A_jR_{\mathbf{X}^{(N)}}\|\leq \frac{C_1}{N^2\Im(z)^4}+\frac{C_2}{N^2}
\end{displaymath}
for some finite constants $C_1,C_2>0$. Invoking the stability of the MDE (see, e.g.,~\cite[Sect.~3.2]{AEKN2019}), it follows that
\begin{displaymath}
\lim_{N\rightarrow\infty}R_{\mathbf{X}^{(N)}}(z)=M(z)
\end{displaymath}
for all $z$ in the upper half-plane. This implies the convergence of the Stieltjes transform of $\E[\mu_{\mathbf{X}^{(N)}}(\Pi_i,\Pi_j)]$. The vague convergence of the underlying measures then follows by adapting the proof of~\cite[Thm.~5.8]{FleermannKirsch2023}. Their weak convergence then follows from the tightness in Lemma \ref{lem-norm_exptight}. In particular, we use~\eqref{eq-complexRadon} to supply the bound needed for the construction of a vaguely convergent subsequence (cf.~\cite[Lem.~2.15]{FleermannKirsch2023}) in the first step. Recalling~\eqref{eq-Stieltjestransf}, the limit is readily identified as $\mu_{ij}$.

\medskip
\underline{\smash{Concentration:}} We show that {for every $\eps>0$,}
\begin{displaymath}
\limsup_{N\rightarrow\infty}\frac{1}{N}\ln[\P(\exists i,j\in[L]: d_{KR}\big(\mu_{\mathbf{X}^{(N)}}(\Pi_i,\Pi_j),\E[\mu_{\mathbf{X}^{(N)}}(\Pi_i,\Pi_j)]\big)\geq\eps)]=-\infty,
\end{displaymath}
which, together with the weak convergence of $\E[\mu_{\mathbf{X}^{(N)}}(\Pi_i,\Pi_j)]$, concludes the proof of Theorem~\ref{thm-muproj.conc}. Let $f$ be a Lipschitz function. We start by noting that, by definition, $\mu_{\mathbf{X}^{(N)}}(\Pi_i,\Pi_j)$ is Lipschitz in the entries of $\sqrt{N}\mathbf{X}^{(N)}$ for the Euclidean norm. More precisely, we have for $\mathbf{H},\mathbf{H}'\in\R^{NL\times NL}$ that
\begin{align}
|\mu_{\mathbf{H}}(\Pi_i,\Pi_j)[f]&-\mu_{\mathbf{H}'}(\Pi_i,\Pi_j)[f]|\leq\frac{1}{N}\Tr(\Pi_i\mathbf{P}^{i\leftrightarrow j}\Pi_j|f(\mathbf{H})-f(\mathbf{H}')|)\nonumber\\
&\qquad\leq\frac{1}{N}\Big(\Tr(\Pi_j\mathbf{P}^{i\leftrightarrow j}\Pi_i)\Big)^{1/2}(\Tr(f(\mathbf{H})-f(\mathbf{H}'))^2)^{1/2}\label{CSa}
\end{align}
by Cauchy-Schwartz's inequality. 
We then observe (see also~\cite{Kittaneh1985,DGH2024}) that if $\mathbf{H}$ (resp.~$\mathbf{H}'$) has eigenvalues $(\lambda_{i})_{1\le i\le N}$ (resp. $(\lambda_{i}')_{1\le i\le N}$)  for the unit eigenvectors $(v_{i})_{1\le i\le N}$ (resp. $(v_{i}')_{1\le i\le N}$), it holds that
\begin{eqnarray*}
\Tr(f(\mathbf{H})-f(\mathbf{H}'))^{2}&=& \sum_{i,j=1}^{N}\langle v_{i},v_{j}'\rangle^{2}\left(f(\lambda_{i})^{2}-f(\lambda_{j}')^{2}-2f(\lambda_{i})f(\lambda_{j}')\right)\\
&\le&\mathrm{Lip}(f)^{2}\sum_{i,j=1}^{N}\langle v_{i},v_{j}'\rangle^{2}\left(\lambda_{i}-\lambda_{j}'\right)^{2}=\mathrm{Lip}(f)^{2}\Tr(\mathbf{H}-\mathbf{H}')^{2}.\end{eqnarray*}
Plugging this inequality into \eqref{CSa} and using that the projections $\Pi_i$ resp. $\Pi_j$, and the permutation matrix $\mathbf{P}^{i\leftrightarrow j}$ have operator norm at most one, we conclude that 
$$|\mu_{\mathbf{H}}(\Pi_i,\Pi_j)[f]-\mu_{\mathbf{H}'}(\Pi_i,\Pi_j)[f]|\leq\frac{1}{\sqrt{N}}\mathrm{Lip}(f)(\Tr(\mathbf{H}-\mathbf{H}')^2)^{1/2}\,.$$
Recalling from~\eqref{eq-model} that the GOE matrices used in defining $\mathbf{X}^{(N)}$ follow the usual scaling, i.e., the entries of $\sqrt{N}\mathbf{X}^{(N)}$ are of order one, and 
using that Lipschitz functions of Gaussian random variables exhibit sub-Gaussian tails (see, e.g.~\cite[Ch.~2]{GIneqBook}), we obtain for any fixed Lipschitz function~$f$ that
\begin{align*}
\P\Big(|\mu_{\mathbf{X}^{(N)}}(\Pi_i,\Pi_j)[f]-\E[\mu_{\mathbf{X}^{(N)}}(\Pi_i,\Pi_j)[f]]|\geq \eps\mathrm{Lip(f)}(\frac{1}{N}\Tr(\Pi_j\mathbf{P}^{i\leftrightarrow j}\Pi_i)\Big)^{1/2}\Big)\leq C_1e^{-C_2\eps^2N^2}
\end{align*}
for suitable constants $C_1,C_2>0$. By approximating any Lipschitz function by a finite basis as in \cite[Corollary 1.4]{GuionnetZeitouni00}, the bound can be generalized to $d_{KR}$, yielding
\begin{align*}
\P\Big(d_{KR}(\mu_{\mathbf{X}^{(N)}}(\Pi_i,\Pi_j),\E[\mu_{\mathbf{X}^{(N)}}(\Pi_i,\Pi_j)]\geq \eps\Big(\frac{1}{N}\Tr(\Pi_j\mathbf{P}^{i\leftrightarrow j}\Pi_i)\Big)^{1/2}\Big)&\leq C_1\eps^{-3/2}e^{-C_2\eps^5 N^2},
\end{align*}
which implies the claim.

\subsection{Proof of Proposition~\ref{prop-tiltedonsphere}}\label{sect-tiltedonsphere}
The key quantity for this proof is the tilted law
\begin{equation}\label{eq-defQ}
\Q_{\mathbf{H},\theta}(\dx \mathbf{u}):=\frac{\re^{\theta NL\langle \mathbf{u},\mathbf{Hu}\rangle}\dx \mathbf{u}}{\int_{\mathds{S}^{NL-1}}\re^{\theta NL\langle \mathbf{u},\mathbf{Hu}\rangle}\dx \mathbf{u}}=\frac{\re^{\theta NL\langle \mathbf{u},\mathbf{Hu}\rangle}}{I_N(\theta,\mathbf{H})}\ \dx \mathbf{u}
\end{equation}
with $\theta\in\R$ and a fixed (deterministic) matrix $\mathbf{H}\in\R^{NL\times NL}$. In this notation, proving Proposition~\ref{prop-tiltedonsphere} is equivalent to showing, again with the definition of $S_\eps$ in \eqref{eq-defSeps},
\begin{equation}\label{eq-tiltedonsphere}
\Q_{\mathbf{H},\theta}\big(S_{\eps}(\lambda_1(\mathbf{H}),\mathbf{v}_1(\mathbf{H}))\big)\geq\re^{-o(N)}
\end{equation}
for $\mathbf{H}\in B_{\delta,N}(x,\Psi)$ with $\lambda_1(\mathbf{H})-3\delta>r_\infty$ and $m_{\mu_\infty}(\lambda_1(\mathbf{H})-3\delta)>-2\theta$, where the error term satisfies $\lim_{N\rightarrow\infty}o(N)/N=0$. We first consider the case where $\lambda_1(\mathbf{H})$ is a sufficiently separated outlier.

\begin{lemma}\label{lem-outlier}
Under the assumptions of Proposition~\ref{prop-tiltedonsphere}, let further $x-\lambda_2(\mathbf{H})>3\delta$. Then~\eqref{eq-tiltedonsphere} holds.
\end{lemma}

\begin{proof}[Proof of Lemma~\ref{lem-outlier}]
We start by noting that the assumptions assure the bound $\lambda_1(\mathbf{H})-\lambda_i(\mathbf{H})>\delta$ {for every $i\ge 2$ }on $\{|x-\lambda_1(\mathbf{H})|\leq\delta\}$. We will first reduce ourselves to Gaussian random variables meaning that we will be able to assume that under $\Q_{\mathbf{H},\theta}$, the distribution of $\Pi^{\perp\mathbf{v}_1(\mathbf{H})}\mathbf{u}$ is approximately the distribution $\widetilde{\Q}$ of the following Gaussian vector :
\begin{displaymath}
\sum_{i=2}^{NL}\frac{\xi_i}{\sqrt{2NL\theta(\lambda_1(\mathbf{H})-\lambda_i(\mathbf{H}))}}\mathbf{v}_i(\mathbf{H})
\end{displaymath}
with $\xi_2,\dots,\xi_{NL}$ i.i.d. $\cN(0,1)$ random variables. We obtain for measurable sets $S$ that
\begin{displaymath}
\Q_{\mathbf{H},\theta}(\Pi^{\perp\mathbf{v}_1(\mathbf{H})}\mathbf{u}\in S)=\frac{\int_{S\cap\mathds{B}_{NL-1}}(1-\|u\|^2)^{-1/2}\widetilde{\Q}(\dx u)}{\int_{\mathds{B}_{NL-1}}(1-\|u\|^2)^{-1/2}\widetilde{\Q}(\dx u)},
\end{displaymath}
where $\mathds{B}_{NL-1}$ denotes the unit ball in $\R^{NL-1}$, and
\begin{equation}\label{eq-titledonspherekey}
\limsup_{N\rightarrow\infty}\frac{1}{N}\ln\Big(\int_{\mathds{B}_{NL-1}}(1-\|u\|^2)^{-1/2}\widetilde{\Q}(\dx u)\Big)\leq0,
\end{equation}
following the proof of~\cite[Prop.~4.3]{CookDG23} and especially (4.17). It only remains to show that
$\widetilde{\Q}(S_\eps(\lambda_1(\mathbf{H}),\mathbf{v}_1(\mathbf{H})))\geq\re^{-o(N)}$. Abbreviate
\begin{displaymath}
Z_{ij}:=\langle P^{i\leftrightarrow j}\Pi_i\Pi^{\perp\mathbf{v}_1(\mathbf{H})}\mathbf{u},\Pi_j\Pi^{\perp\mathbf{v}_1(\mathbf{H})}\mathbf{u}\rangle,
\end{displaymath}
and let $\widetilde{\E}$ denote the expectation under $\widetilde{\Q}$. It follows that
\begin{displaymath}
\widetilde{\E}[Z_{ij}]=\sum_{a=2}^{NL}\frac{1}{2NL\theta(\lambda_1(\mathbf{H})-\lambda_a(\mathbf{H}))}\langle P^{i\leftrightarrow j}\Pi_i\mathbf{v}_a(\mathbf{H}),\Pi_j\mathbf{v}_a(\mathbf{H})\rangle,
\end{displaymath}
which is, up to some constants, equal to the Stieltjes transform of $\mu_\mathbf{H}(\Pi_i,\Pi_j)$. More precisely, setting $f_{y}(x) = \frac{1}{y -x}\wedge (2 \delta^{-1})$ and remembering that $\delta >0$ is a lower bound of the spectral gap, we have that 

\[ \widetilde{\E}[Z_{ij}] = \frac{1}{2}\int f_{\lambda_1(\mathbf{H})}(x) d \mu_\mathbf{H}(\Pi_i,\Pi_j)(x) - \frac{\delta}{2NL \theta} \]

When $N$ goes to $+\infty$, since $\mathbf{H}\in \Omega_N$, $\mu_{\mathbf{H}}(\Pi_i,\Pi_j)$ converges weakly to $\mu_{i,j}$ and for $N$ large enough, its support is included in $[-\mathcal{C}, x - 2\delta ]$. Therefore, for $N$ large enough, we can write

\[ \int f_{\lambda_1(\mathbf{H})}(x) d \mu_{i,j}(x) =\int \frac{1}{y - x} d \mu_{i,j}(x) =\phi(\theta, \lambda_1(\mathbf{H}))_{i,j} \]
Furthermore, since for $y >x-\delta$ the set of functions $f_{y}$ restricted on the interval $[-\mathcal{C}, x - 2\delta ]$ is uniformly Lipschitz. For $\mathbf{H} \in \Omega_N$ defined by equation \eqref{eq-defOmegaN}, we have that $\widetilde{\E}[Z_{ij}]$ is  uniformly close to $\phi(\theta, \lambda_1(\mathbf{H}))_{i,j}$.
Therefore we have
\begin{displaymath}
\lim_{N\rightarrow\infty}\Big|\widetilde{\E}[Z_{i,j}]-\varphi(\theta,\lambda_1(\mathbf{H}))_{ij}\Big|=0
\end{displaymath}
for all $i,j\in[L]$.

\medskip
By construction, $\Pi^{\perp\mathbf{v}_1(\mathbf{H})}\mathbf{u}$ is a Gaussian vector whose covariance matrix $\Sigma$ is bounded in spectral radius by $1/(2\theta NL\delta)$. Using Wick's theorem, we compute
\begin{displaymath}
\Var[Z_{ij}]=\Tr[\Sigma_{ii}\Sigma_{jj}]+\Tr[\Sigma_{ij}^2]\leq\frac{1}{2NL\theta^2\delta^2},
\end{displaymath}
where $\Sigma_{ij}$ denotes the part of the covariance matrix describing the covariance between the nonzero elements of $\Pi_i\Pi^{\perp\mathbf{v}_1(\mathbf{H})}\mathbf{u}$ and $\Pi_j\Pi^{\perp\mathbf{v}_1(\mathbf{H})}\mathbf{u}$. Hence, by Chebyshev's inequality and a union bound, it follows that
\begin{equation}\label{eq-Sepsbound1}
\widetilde{\Q}\Big(\exists i,j\in[L]: \big|Z_{ij}-\widetilde{\E}[Z_{ij}]\big|>\frac{\eps}{2}\Big)\leq\frac{4L^2}{\eps^2}\max_{i,j}\Var[Z_{i,j}]\leq\frac{2L}{N\theta^2\delta^2\eps^2}.
\end{equation}
Recalling that $\mathbf{v}_1(\mathbf{H})$ is a (fixed) unit vector, the definition of $\rho$ in~\eqref{eq-defrho}, and the fact that the entries of $\Pi^{\perp\mathbf{v}_1(\mathbf{H})}\mathbf{u}$ are centered Gaussians, we further note that
\begin{displaymath}
\Var[\rho(\Pi^{\perp\mathbf{v}_1(\mathbf{H})}\mathbf{u},\mathbf{v}_1(\mathbf{H}))_{ij}]\leq \frac{4}{NL\theta\delta}.
\end{displaymath}
By Chebyshev's inequality and a union bound, we thus obtain
\begin{equation}\label{eq-Sepsbound2}
\widetilde{\Q}\Big(\exists i,j\in[L]: \big|\rho(\Pi^{\perp\mathbf{v}_1(\mathbf{H})}\mathbf{u},\mathbf{v}_1(\mathbf{H}))_{ij}\big|\geq\eps\Big)\leq \frac{4L}{N\theta\delta \eps^2}.
\end{equation}
Let now $N$ be large enough that $|Z_{ij}-\varphi(\theta,\lambda_1(\mathbf{H}))|<\eps/2$ for all $i,j\in[L]$. Then the bounds in~\eqref{eq-Sepsbound1} and~\eqref{eq-Sepsbound2} are sufficient to control the probability of $S_\eps(\lambda_1(\mathbf{H}),\mathbf{v}_1(\mathbf{H}))$ as desired. More precisely,
\begin{displaymath}
\widetilde{\Q}(S_\eps(\lambda_1(\mathbf{H}),\mathbf{v}_1(\mathbf{H})))\geq 1-\frac{1}{N}\Big(\frac{2L}{\theta^2\delta^2\eps^2}+\frac{4L}{\theta\delta\eps^2}\Big),
\end{displaymath}
which, combined with~\eqref{eq-titledonspherekey} yields the claim.
\end{proof}

We complete the proof of Proposition~\ref{prop-tiltedonsphere} by considering the case where $\lambda_1(\mathbf{H})$ is not an outlier. To reduce the general case to the one treated in Lemma~\ref{lem-outlier}, we introduce the following notation similar to~\cite{DGH2024}. Let $V$ be a Euclidean vector space and $T$ a self-adjoint automorphism of $V$. In this setting, we define
\begin{equation}\label{eq-defQV}
\Q_{T,\theta}^V(\dx u):=\frac{\re^{\dim(V)\theta\langle u,Tu\rangle}}{\E\big[\re^{\dim(V)\theta\langle u,Tu\rangle}\big]}\ \dx u,
\end{equation}
where $u$ is taken from $\mathds{S}^{NL-1}\cap V$. Note that the setting $V=\R^{NL}$ in~\eqref{eq-defQV} reproduces~\eqref{eq-defQ}. Next, recall the definition of $\Omega_N$ in~\eqref{eq-defOmegaN} for some fixed constant $\cC>0$ and a sequence $(\eps_N)_N$ of positive numbers with $\lim_{N\rightarrow\infty}\eps_N=0$. Picking a sequence $(k_N)_N$ of integers with $\lim_{N\rightarrow\infty}k_N/N=0$, we define
\begin{equation}\label{eq-defOmegaNprime}
\Omega_N':=\Omega_N\cap\{\lambda_{k_N}(\mathbf{H})<x-2\delta\}.
\end{equation}
 Note that under the spectral gap assumption in Lemma~\ref{lem-outlier}, we may simply choose $k_N=2$. In the general case, we can follow the argument in~\cite{DGH2024} to find a sequence $(k_N)_N$ such that { $k_N/N$ goes to zero as $N$ goes to infinity and so that  $$
B_{\delta,\kappa, N}(x,\Psi)
\subset \Omega_N'\cap\{|\lambda_1(\mathbf{H})-x|<\delta\},
$$}
where $B_{\delta,\kappa,N}(x,\Psi)$ was defined in~\eqref{eq-defBdeltaN}. As a consequence, it is enough to show Proposition~\ref{prop-tiltedonsphere} only for $\mathbf{H}\in\Omega_N'$. Lastly, we introduce $\Pi_{\mathrm{out}}$ as the orthogonal projector onto $\mathrm{span}\{\mathbf{v}_2(\mathbf{H}),\dots,\mathbf{v}_{k_N}(\mathbf{H})\}$ and $\Pi_{\mathrm{cut}}=I_{NL}-\Pi_{\mathrm{out}}$ that is the orthogonal projector onto
\begin{displaymath}
V:=\mathrm{span}\{\mathbf{v}_1(\mathbf{H}),\mathbf{v}_{k_N+1}(\mathbf{H}),\dots,\mathbf{v}_{NL}(\mathbf{H})\},
\end{displaymath}
and we further define
\begin{equation}\label{eq-defcut}
\mathbf{u}_{\mathrm{cut}}:=\frac{\Pi_{\mathrm{cut}}\mathbf{u}}{\|\Pi_{\mathrm{cut}}\mathbf{u}\|},\quad \mathbf{H}^{\mathrm{cut}}:=\frac{NL^2}{NL-k_N+1}\mathbf{H}|_V
\end{equation}
with $\mathbf{u}\in\R^{NL}$ denoting a unit vector and $\mathbf{H}|_V$ denoting the matrix $\mathbf{H}$ restricted to $V$. We will see that most realizations $\mathbf{H}^{\mathrm{cut}}$ have, by construction, a spectral gap that satisfies the additional hypothesis in Lemma~\ref{lem-outlier} while the remainder is negligible. This is the core of the proof of Proposition~\ref{prop-tiltedonsphere}.

\begin{proof}[Proof of Proposition~\ref{prop-tiltedonsphere}]
Similar to~\eqref{eq-defmuproj}, we introduce
\begin{align*}
\nu_{\mathbf{H}}(\Pi_i,\Pi_j)&=\frac{L}{NL-k_N+1}\Big(\sum_{a=k_N+1}^{NL}\langle \mathbf{P}^{i\leftrightarrow j}\Pi_i\mathbf{v}_a(\mathbf{H}),\Pi_j\mathbf{v}_a(\mathbf{H})\rangle \delta_{\lambda_a(\mathbf{H})}\\
&\qquad\qquad\qquad\qquad\qquad+\langle \mathbf{P}^{i\leftrightarrow j}\Pi_i\mathbf{v}_1(\mathbf{H}),\Pi_j\mathbf{v}_1(\mathbf{H})\rangle\delta_{\lambda_1(\mathbf{H})}\Big) \end{align*}
for $\mathbf{H}\in\R^{NL\times NL}$. By construction, we have for $\mathbf{H}\in\Omega_N'$ that
\begin{displaymath}
\lim_{N\rightarrow\infty}\nu_{\mathbf{H}}(\Pi_i,\Pi_j)=\mu_{ij}\quad a.s.,
\end{displaymath}
and we can even find a sequence $(\eps_N'')_N$ with $\lim_{N\rightarrow\infty}\eps_N''=0$ such that
\begin{displaymath}
d_{KR}(\nu_{\mathbf{H}}(\Pi_i,\Pi_j),\mu_{ij})<\eps_N''
\end{displaymath}
with large probability. Moreover, $\|\mathbf{H}^{\mathrm{cut}}\|\leq2L\cC$ for $N$ large enough, where $\cC$ is the constant used to define $\Omega_N$ in~\eqref{eq-defOmegaNprime}.

\medskip
Lastly, for $N\in\N$, let $V\subset\R^{NL}$ denote a linear subspace with $\dim(V)=NL-k_N+1$ and set
\begin{displaymath}
\Omega_N'':=\bigcup_{V\subset\R^{NL}\text{ subspace}}\{\mathbf{H}\in\mathrm{Sym}_{NL}(V):\|\mathbf{H}\|\leq 2L\cC, \max_{i,j\in[L]}d_{KR}(\mu_{\mathbf{H}}(\Pi_i,\Pi_j),\mu_{i,j})\leq\eps''\}.
\end{displaymath}
In this setup, we have that $\mathbf{H}\in\Omega_N'$ for $\Omega_N'$ defined in~\eqref{eq-defOmegaNprime} implies that the corresponding~$\mathbf{H}^{\mathrm{cut}}$ in~\eqref{eq-defcut} lies in $\Omega_N''$ for $N$ large enough. Moreover, $|\lambda_1(\mathbf{H})-x|\leq\delta$ for $\mathbf{H}\in\Omega_N'$ implies that $|\lambda_1(\mathbf{H}^{\mathrm{cut}})-x|\leq\delta$ for $N$ large enough and for all $\mathbf{H}\in\Omega_N'$, the corresponding $\mathbf{H}^{\mathrm{cut}}$ satisfies $\lambda_2(\mathbf{H}^{\mathrm{cut}})\leq x-2\delta$ for $N$ large enough, i.e., $\mathbf{H}^{\mathrm{cut}}$ satisfies the spectral gap assumption in Lemma~\ref{lem-outlier}.

\medskip
Applying Lemma~\ref{lem-outlier} for $\mathbf{H}^{\mathrm{cut}}$ now yields that if
\begin{displaymath}
S:=\{\mathbf{u}\in\mathds{S}^{NL-1}\cap V: \|\rho(\Pi^{\perp\mathbf{v}_1(\mathbf{H})}\mathbf{u})-\varphi(\lambda_1(\mathbf{H}^{\mathrm{cut}}),\theta)\|<\eps,\|\rho(\Pi^{\perp\mathbf{v}_1(\mathbf{H})}\mathbf{u},\mathbf{v}_1(\mathbf{H}^{\mathrm{cut}})\|<\eps\},
\end{displaymath}
we have
\begin{equation}\label{eq-cutpart}
\frac{1}{N}\ln\big(\Q^V_{\mathbf{H}^{\mathrm{cut}},\theta}(S)\big)\geq o_N(1),
\end{equation}
where $\Q^V_{\mathbf{H}^{\mathrm{cut}},\theta}$ is defined using~\eqref{eq-defQV}. To bound the remainder, we use the following analog to~\cite[Lem.~3.7]{DGH2024}. Its proof relies on Lemma~\ref{lem-convsphint} and is, up to adjusting the dimension of the quantities considered, identical to~\cite{DGH2024}. We omit the details.

\begin{lemma}\label{lem-outpart}
Let $\delta>0$, $x>r_\infty+3\delta$, and $-2\theta<m_{\mu_\infty}(x-3\delta)$. Then for any Borelian $B\subset\R^{NL}$, $\eps>0$, and $\mathbf{H}\in\Omega_N'$ with $|\lambda_1(\mathbf{H})-x|\leq \delta$ we have that
\begin{align*}
&\frac{1}{N}\ln\big(\Q_{\mathbf{H},\theta}(\|\Pi_{\mathrm{out}}\mathbf{u}\|^2\leq\eps,\mathbf{u}_{\mathrm{cut}}\in B)\big)\\
&\geq\frac{1}{N}\ln\big(\P_\mathbf{u}(\|\Pi_{\mathrm{out}}\mathbf{u}\|^2\leq\eps)\big)+\frac{1}{N}\ln\big(\Q^V_{\mathbf{H}^{\mathrm{cut}},\theta}(\mathbf{u}\in B)\big)+\cO(\cC\theta\eps)+\rem(N)
\end{align*}
with an error $\rem(N)$ that satisfies $\lim_{N\rightarrow\infty}\rem(N)=0$.
\end{lemma}

We use this bound as follows: It is readily checked that for any $\eps>0$, we have
\begin{displaymath}
\lim_{N\rightarrow\infty}\frac{1}{N}\ln\big(\P_\mathbf{u}(\|\Pi_{\mathrm{out}}\mathbf{u}\|^2\leq\eps)\big)=0.
\end{displaymath}
Hence, using Lemma~\ref{lem-outpart} and~\eqref{eq-cutpart}, we can find a sequence $(\eps_N''')_N$ with $\lim_{N\rightarrow\infty}\eps_N'''=0$ such that
{\begin{displaymath}
\Q_{\mathbf{H},\theta}(\|\Pi_{\mathrm{out}}\mathbf{u}\|^2\leq\eps_N''',\mathbf{u}_{\mathrm{cut}}\in S)\geq\re^{-o(N)},
\end{displaymath}}
where $\lim_{N\rightarrow\infty}o(N)/N=0$. To conclude, observe that $S$ was constructed such that
\begin{displaymath}
\{\mathbf{u}\in\mathds{S}^{NL-1}:\|\Pi_{\mathrm{out}}\mathbf{u}\|^2\leq\eps,\mathbf{u}_{\mathrm{cut}}\in S\}\subset S_{2\eps}\big(\mathbf{v}_1(\mathbf{H}),\lambda_1(\mathbf{H})\big)
\end{displaymath}
for $N$ large enough. Recall that $S_{\eps}(\mathbf{v}_1(\mathbf{H}),\lambda_1(\mathbf{H}))$ was defined in~\eqref{eq-defSeps}. Putting everything together, it now follows that
\begin{displaymath}
\Q_{\mathbf{H},\theta}(S_{2\eps}\big(\mathbf{v}_1(\mathbf{H}),\lambda_1(\mathbf{H})\big))\geq\re^{-o(N)}
\end{displaymath}
with $\lim_{N\rightarrow\infty}o(N)/N=0$ as desired. This completes the proof of Proposition~\ref{prop-tiltedonsphere}.
\end{proof}

\section{The condition $\langle \Psi,\cS\Psi\rangle>\eps$}
With the proof of the weak LDP upper bound completed, we now turn to the lower bound. To ensure that the bounds match (and thus yield the desired weak LDP when put together), we first consider the following improvement over Proposition~\ref{prop-weakLDPsup} that introduces an additional restriction in the infimum over $\Psi$. Recall that $\cS[T]=\sum_{j=1}^kA_jTA_j$ for $T\in\R^{L\times L}$ and $\langle \Psi,\cS[\Psi]\rangle=\Tr[\Psi^T\cS[\Psi]]$ with $\Psi\in\mathrm{Sym}^+_L(\R)$. We further set $r_0:= \lambda_1(A_0)$.
\begin{proposition}\label{prop-improved_weakLDPsup}
For every $x>r_\infty$ there exists $\eps_x>0$ such that
\begin{displaymath}
\limsup_{\delta\rightarrow 0}\lim_{N\rightarrow\infty}\frac{1}{N}\ln\big(\P(|\lambda_1-x|\leq\delta)\big)\leq-\inf_{\Psi:\langle\Psi,\cS[\Psi]\rangle\geq\eps_x}\sup_{\theta>0}\cF(\theta,x,\Psi).
\end{displaymath}
Furthermore, one can choose $\epsilon_x$ such that for every $r$ such $r > r_0=\lambda_1(A_0)$ and $r > r_{\infty}$ and every $R > r_{\infty}$
\[ \inf_{x \in [r, R]} \epsilon_x >0. \]
\end{proposition}

The key ingredient to the proof of Proposition~\ref{prop-improved_weakLDPsup} is the fact that we can exclude the eigenvectors $\mathbf{v}_1$ for $\lambda_1$ such that $\langle\rho(\mathbf{v}_1),\cS[\rho(\mathbf{v}_1)]\rangle$ is small. For $\eps>0$, define the set
\begin{equation}\label{eq-defsetrho}
S_{\rho,\eps}:=\{\mathbf{u}\in\mathds{S}^{NL-1}:\langle\rho(\mathbf{u}),\cS[\rho(\mathbf{u})]\rangle\leq\eps \}.
\end{equation}
We then have the following bound.

\begin{lemma}\label{lem-weakLDPsupimprovement1} There exists a constant $C>0$ such that for all $M>0$ and $r>0$ it holds that
\begin{align*}
\frac{1}{N}&\ln(\P[\mathbf{X}^{(N)}\in\mathrm{Sym}_{NL}(\R):\|\mathbf{X}^{(N)}\|\leq M, \lambda_1(\mathbf{X}^{(N)})>r + r_0,\mathbf{v}_1(\mathbf{X}^{(N)})\in S_{\rho,\eps}])\\
&\leq C+L\ln(M/r)- \frac{r^2}{ \eps} .
\end{align*}
\end{lemma}

\begin{proof}
Fix $\eps>0$ and construct a $\frac{r}{4M}$-net of $S_{\rho,\eps}$ in~\eqref{eq-defsetrho}. We denote this set as $\cR_{N,\eps}$. Note that $\cR_{N,\eps}$ can be chosen such that
\begin{displaymath}
|\cR_{N,\eps}|\leq\Big(\frac{C'M}{r}\Big)^{NL}
\end{displaymath}
with a constant $C'$ that is independent of the parameters $N$, $M$, and $r$. Next, estimate
\begin{align}
&\P[\mathbf{H}\in\mathrm{Sym}_{NL}(\R):\|\mathbf{H}\|\leq M, \lambda_1(\mathbf{H})>r,\mathbf{v}_1(\mathbf{H})\in S_{\rho,\eps}]\NN\\
&\leq\P[\mathbf{H}\in\mathrm{Sym}_{NL}(\R):\|\mathbf{H}\|\leq M,\exists \mathbf{u}\in S_{\rho,\eps}\text{ with }\langle\mathbf{u},\mathbf{H}\mathbf{u}\rangle\geq r]\NN\\
&\leq\P[\mathbf{H}\in\mathrm{Sym}_{NL}(\R):\|\mathbf{H}\|\leq M,\exists \mathbf{u}\in \cR_{N,\eps}\text{ with }\langle\mathbf{u},\mathbf{H}\mathbf{u}\rangle\geq r/2]\NN\\
&\leq \Big(\frac{C'M}{r}\Big)^{NL}\max_{\mathbf{u}\in S_{\rho,\eps}}\P[\mathbf{H}\in\mathrm{Sym}_{NL}(\R):\langle\mathbf{u},\mathbf{H}\mathbf{u}\rangle\geq r/2]\label{eq-netarg_innerprod}
\end{align}
where we introduced the $\frac{r}{4M}$-net $\cR_{N,\eps}$ in the second step and then used the bound on $|\cR_{N,\eps}|$. Observe that the last expression in~\eqref{eq-netarg_innerprod} only concerns the tail of the random variables $\langle \mathbf{u},\mathbf{X}^{(N)}\mathbf{u}\rangle$ for fixed $\mathbf{u}\in S_{\rho,\eps}$. Writing out~\eqref{eq-model} in the inner product, we get
\begin{align*}
\langle \mathbf{u},\mathbf{X}^{(N)}\mathbf{u}\rangle&=\sum_{j=1}^k\Big[\sum_{a=1}^N(W_j)_{aa}\sum_{c,d=1}^L(A_j)_{cd}(u_c)_a(u_d)_a+\sum_{a<b}(W_j)_{ab}\sum_{c,d=1}^L(A_j)_{cd}(u_c)_a(u_d)_b\Big]\\
&\quad+\sum_{c,d=1}^L(A_0)_{c,d}\langle u_c,u_d\rangle,
\end{align*}
which is again a Gaussian random variable. We recall $r_0=\lambda_1(A_0)$, and estimate
\begin{displaymath}
\E[\langle \mathbf{u},\mathbf{X}^{(N)}\mathbf{u}\rangle]=\sum_{c,d=1}^L(A_0)_{c,d}\langle u_c,u_d\rangle = \Tr[ A_0 \Psi]\leq r_0.
\end{displaymath}
Since $\mathbf{u}\in S_{\rho,\eps}$, we can further bound the variance by
\begin{align*}
\Var[\langle \mathbf{u},\mathbf{X}^{(N)}\mathbf{u}\rangle]
&=\frac{2}{N}\sum_{j=1}^k\Big[\sum_{c,d=1}^L(A_j)_{cd}(A_j)_{c'd'}\langle u_c,u_c'\rangle\langle u_d,u_d'\rangle\Big]\\
&=\frac{2}{N}\langle \Psi,\cS[\Psi]\rangle\leq\frac{2\eps}{N}
\end{align*}
Applying a classical Gaussian tail bound, we thus find $c>0$ so that
\begin{displaymath}
\P[\langle\mathbf{u},\mathbf{X}^{(N)}\mathbf{u}\rangle -r_0 \geq r/2] \leq \P[\langle\mathbf{u},\mathbf{X}^{(N)}\mathbf{u}\rangle -\mathbb E[ \langle \mathbf{u}, \mathbf{X}^{(N)} \mathbf{u}\rangle] \geq r/2]\leq  e^{- c r^2 N/\eps}
\end{displaymath}
which, together with~\eqref{eq-netarg_innerprod}, yields the desired bound.
\end{proof}

Combining Lemma~\ref{lem-weakLDPsupimprovement1} with Lemma~\ref{lem-norm_exptight} allows to remove the condition $\|\mathbf{H}^{(N)}\|\leq M$ on the left-hand side. Recall that $r_\infty$ denotes the right end of the support of the limiting spectral measure $\mu_\infty$ (cf. Lemma~\ref{lem-convlargestev}).

\begin{lemma}\label{lem-weakLDPsupimprovement2}
There exist constants $C,C'>0$ such that for any $M>0$ and $x>r_\infty$ it holds that
\begin{align*}
\frac{1}{N}&\ln(\P[\mathbf{X}^{(N)}\in\mathrm{Sym}_{NL}(\R):|\lambda_1(\mathbf{X}^{(N)})-x|\leq\delta,\mathbf{v}_1(\mathbf{X}^{(N)})\in S_{\rho,\eps})\\
&\leq \max\{ C'+L\ln(M/r)- \frac{r^2}{ \eps} ,  C - \frac{L M}{4} \}.
\end{align*}
\end{lemma}

The last tool needed for the proof of Proposition~\ref{prop-improved_weakLDPsup} is the following bound on the rate function that we prove in Section~\ref{sect-functFpart1}. Recall that $\rate(x,\eps)$ was defined in~\eqref{eq-defapproxI}.

\begin{lemma}\label{lem-Ibound}{
There exists some $\eta >0$ and $\gamma>0$ (depending only on the matrices $A_0,\dots,A_k$) such that, for every $x>r_\infty$ and every $\epsilon \in (0,\eta)$ it holds that }
\begin{displaymath}
\rate(x,\epsilon) \leq  \gamma(x^2+1).
\end{displaymath}
\end{lemma}
The proof of this lemma is delayed to Appendix~\ref{sect-functFpart1}. We can now prove Proposition~\ref{prop-improved_weakLDPsup} using Lemmas~\ref{lem-weakLDPsupimprovement2} and~\ref{lem-Ibound}.

\begin{proof}[Proof of Proposition \ref{prop-improved_weakLDPsup}]
Given $x > r_{\infty}$, we first choose $M$ large enough so that $\frac{LM }{4} -C > \gamma(x^2+1) +1 $. With such a fixed $M$ we then choose $\epsilon >0$ small enough such that
\begin{displaymath}
- C'- L \ln(M/|x- r_0|) + \frac{|x- r_0|^2}{\epsilon} > \gamma(x^2+1)  +1,
\end{displaymath}
and call it $\epsilon_x $. It is easy to see that one can choose $\epsilon_x$ in a way that 
{$\epsilon_x$ stays uniformly bounded from below by a positive constant $\eps$ for $x\in{[r,R]}$.}
Now, by definition, we have that, { with the notation of \eqref{eq-defsetrho}, for $N$ large enough}

\[ \frac{1}{N}\ln(\P[|\lambda_1(\mathbf{X}^{(N)})-x|\leq\delta,\mathbf{v}_1(\mathbf{X}^{(N)})\in S_{\rho,\eps_x}])< - \rate_1(x,\epsilon_x) - 1.\]
As a consequence, 
\begin{align*}
\P[ |\lambda_1(\mathbf{X}^{(N)})-x|\leq\delta ] & \leq \P[ |\lambda_1(\mathbf{X}^{(N)})-x|\leq\delta, \mathbf{v}_1(\mathbf{X}^{(N)})\in S_{\rho,\eps}) ]\\ &\quad+ \P[ |\lambda_1(\mathbf{X}^{(N)})-x|\leq\delta, \mathbf{v}_1(\mathbf{X}^{(N)})\in S_{\rho,\eps}^c ) ].
\end{align*}

Now, as we  just saw, the first term can be bounded by $e^{ -N( \rate(x,\epsilon){+} 1) } $. The second term can be bounded using Proposition~\ref{prop-weakLDPsup} and  a covering argument of the compact set
$\mathcal{T}_{\epsilon}:=\{ \Psi \in \mathrm{Sym}^{+,1}_L( \R) : \Tr[ \Psi^T \mathcal{S}(\Psi)] \geq \epsilon \}$ similar to the one used to prove \eqref{coveri}. We ultimately get
\begin{align*}
\limsup_{N \to \infty} \frac{1}{N} \ln \P[ |\lambda_1(\mathbf{X}^{(N)})-x|\leq\delta, \mathbf{v}_1(\mathbf{X}^{(N)})\in S_{\rho,\eps}^c ) ] &\leq - \inf_{\Psi \in \mathcal{T}_{\epsilon} }  \sup_{\theta \geq 0 }\mathcal{F}(\theta,x, \Psi) \\
&\leq {-} \rate( x, \epsilon) 
\end{align*}
where we used the definition of $\rate(x,\epsilon)$ in~\eqref{eq-defapproxI}. Putting the two estimates together, we get 
\begin{align*}
\lim_{\delta \to 0} \limsup_{N \to \infty} \frac{1}{N} \ln \P[|\lambda_1(\mathbf{X}^{(N)})-x|\leq\delta ] \leq  {-} \rate( x, \epsilon) 
\end{align*}
as desired.
\end{proof}

\section{Proof of Theorem~\ref{thm-main}: Weak Large Deviation Lower Bound}\label{sect-wLDPinf}
The goal of this section is to prove the following lower bound.
\begin{proposition}[Lower bound in Theorem~\ref{thm-mainweak}]\label{prop-weakLDPinf}
For every $x>r_\infty$, we have
\begin{displaymath}
\lim_{\delta\rightarrow0}\liminf_{N\rightarrow\infty}P_\delta^{(N)}(x)\geq\sup_{\substack{\Psi\in\mathrm{Sym}_L^{+,1}(\R)\\\Tr[\Psi^T \cS( \Psi)]\neq 0}}\inf_{\theta\geq0} -\cF(\theta,x,\Psi,1),
\end{displaymath}
where $P_{\delta}^{(N)}$ is defined in~\eqref{eq-defPNdelta}.
\end{proposition}

The key to the proof of Proposition~\ref{prop-weakLDPinf} lies in studying the typical random matrices under the tilted measure. We observe that, in this setting, the law of $\mathbf{X}^{(N)}$ is equal to that of the original matrix up to a perturbation that we can write down explicitly.

\begin{lemma}
	Let $ \theta >0$ and $\mathbf{u} \in \mathbb{S}^{NL -1}$. We consider the tilted measure $\P^{(\mathbf{u},\theta)}$ defined by 
	
	\[ d\P^{(\mathbf{u},\theta)}( \mathbf{X}^{(N)}) = \frac{\exp(N \theta \langle \mathbf{u},\mathbf{X}^{(N)}\mathbf{u} \rangle )}{\E_W[ \exp(N \theta \langle \mathbf{u},\mathbf{X}^{(N)}\mathbf{u} \rangle )] }d \P (\mathbf{X}^{(N)}). \]
	
	Under this tilted measure, the matrix $\mathbf{X}^{(N)}$  can be decomposed in the following way 
	
	\[ \mathbf{X}^{(N)} =  \widetilde{\mathbf{X}}^{(N)} + 2 \theta T S T^*   \]
	
	where $\widetilde{\mathbf{X}}^{(N)}$ has the same distribution as $\mathbf{X}^{(N)}$ under $\P$, $T$ is a $NL \times L^2$ matrix and $S$ is a $L^2 \times L^2$ matrix. For simplicity, we will index the set  $[1,L^2]$ using double indexes of the form $ab$ where $a,b \in [1,L]$. For such $a,b$, the $ab$ column of $T$ is $e_a \otimes u_b$ where $(e_1, \dots , e_L)$ is the canonical base of $\R^L$. Furthermore
\[ S_{ab,cd} = \sum_{j \geq 1} (A_j)_{ac} (A_j)_{bd}. \]
or in other terms $S = \sum_{j \geq 1} A_j \otimes A_j$. 
	\end{lemma}

\begin{proof}
To prove this, let us write down and expand the scalar product as

\begin{align*}
\langle \mathbf{u},\mathbf{X}^{(N)}\mathbf{u} \rangle &= \sum_{j} \sum_{a,b} (A_j)_{a,b} \langle u_a W_j^{(N)}, u_b \rangle + \sum_{a,b} (A_0)_{a,b} \langle u_a, u_b  \rangle \\
&= \sum_{j} \sum_{a,b} (A_j)_{a,b} \text{Tr}  ((u_b)^* u_a W_j^{(N)}) + \sum_{a,b} (A_0)_{a,b} \langle u_a, u_b  \rangle. 
\end{align*}

Therefore, classical properties of Gaussian random variables give us that under the tilted measure $\P^{(u,\theta)}$, we have for $j \geq 1$ that

\[ W_j^{(N)} = \widetilde{W}_j^{(N)} + 2 \theta \sum_{a,b} (A_j)_{ab} (u_b)^* u_a, \]

where the $\widetilde{W}_j^{(N)}$ are GOE under $\P^{(u,\theta)}$. As a consequence, we can write
$$ \mathbf{X}^{(N)} =  \widetilde{\mathbf{X}}^{(N)} + 2 \theta D,$$
where $\widetilde{\mathbf{X}}^{(N)} = \sum_{j=1}^k A_j\otimes \widetilde{W}_j^{(N)}+A_0\otimes \Id_N $
and 
\begin{align*}
	D &= \sum_{j} A_j  \otimes \sum_{a,b} (A_j)_{ab} (u_b)^* u_a \\
	&=   \sum_{j} \Big(  \Big(\sum_{c,d} (A_j)_{cd} e_c (e_d)^* \Big)  \otimes \Big(\sum_{a,b} (A_j)_{ab} u_a(u_b)^* \Big) \Big) \\
	&=  \sum_{c,d}\sum_{a,b} \Big( \sum_{j} (A_j)_{cd}(A_j)_{ab} \Big) (e_c  \otimes   u_a)( e_d \otimes u_b)^* \\
		&=  \sum_{c,a,d,b}S_{ca,db} (e_c  \otimes   u_a)( e_d \otimes u_b)^* = T S T^*. 
\end{align*}
This is the claim.
\end{proof}

Having established the structure of $\mathbf{X}^{(N)}$ under the tilted measure, we can use classical results on outliers of perturbed random matrices to identify the limit of $(\lambda_1(\mathbf{X}^{(N)}))_N$ as~$N\rightarrow\infty$.

\begin{lemma}\label{lem-theeq}
Let $ \theta >0$ and let us consider a sequence of vectors $(\mathbf{u}^{(N)})_{N \in \N}$ such that for all $N$, $\mathbf{u}^{(N)}  \in \mathbb{S}^{NL -1}$ and such that for all  $N$, the profile of $\mathbf{u}^{(N)}$ converges towards  a matrix $\Psi$ as $N\rightarrow\infty$. Let us consider the following equation on $z\in\R$ 

\begin{equation}\label{theeq}
\det\left( \Id_{L^2} + 2 \theta  S ( M(z) \otimes \Psi) \right) = 0
\end{equation} 
where the $L \times L$ matrix $M(z)$ was introduced in~\eqref{eq-Stieltjestransf}.
Then, if  the equation \eqref{theeq} has at least one solution in $]r_\infty , + \infty[$, the largest eigenvalue $\lambda_1 ( \mathbf{X}^{(N)})$ converges toward the largest solution of this equation when $N$ goes to $\infty$, almost surely under the measure $\P^{(\mathbf{u}^{(N)},\theta)}$.
\end{lemma}
	
\begin{proof}
By the almost sure convergence of the empirical measure towards $\mu_\infty$ under the measure $\P^{(\mathbf{u}^{(N)},\theta)}$, the largest eigenvalue of $\mathbf{X}^{(N)} $ is larger or equal to $r_\infty$. Moreover, it is given by the largest solution to the equation
		
\[ \det( z  - \mathbf{X}^{(N)} )= 0.\]
Let us assume that $z$ is an outlier, namely that it is greater than $r_\infty+\eps$ for some $\eps>0$ independent of $N$. Using that $\mathbf{X}^{(N)}  = \widetilde{\mathbf{X}}^{(N)} + 2 \theta T S T^* $, one can rewrite this equation in the following way 

\[ \det( z - \widetilde{\mathbf{X}}^{(N)} ) \det( \Id_{N} -  2 \theta TS T^*( z - \widetilde{\mathbf{X}}^{(N)})^{-1}) = 0 .\]

The product expression above is a degree $N$ polynomial. Furthermore when $N$ goes to infinity, the eigenvalues of $\widetilde{\mathbf{X}}^{(N)}$ are bounded above by $r_\infty+\eps/2$  almost surely by Lemma~\ref{lem-convlargestev}. Therefore, to identify the positions of the outliers of $\mathbf{X}^{(N)}$, one has to study the zeroes of the limit of the function $z \mapsto \det( \Id_{N} - 2\theta TS T^*( z -\widetilde{\mathbf{X}}^{(N)} )^{-1})$ on $]r_\infty, + \infty[$.  
\medskip
For this we will use the following classical matrix identity: if $A,B$ are two matrices of respective dimensions $N\times M$ and $M \times N$, $\det( \Id_N + AB) = \det( \Id_M + BA)$. Therefore, we have 
	
\[\det( \Id_{N} -  2 \theta TS T^*( z - \widetilde{\mathbf{X}}^{(N)})^{-1}) = \det( \Id_{L^2} -  2 \theta S T^*( z - \widetilde{\mathbf{X}}^{(N)})^{-1}T) . \]

and it suffices to prove that almost surely, $\lim_{N \to  \infty}T^*( z - \widetilde{\mathbf{X}}^{(N)})^{-1}T = - M(z) \otimes \Psi$ for every $z$ in some arbitrary complex neighborhood of $] r_\infty + \epsilon, + \infty[$ disjoint from $\text{supp}(\mu_\infty)$. Indeed, if we recall that with high probability, the largest eigenvalue of $\mathbf{X}^{(N)}$ remains bounded (which prevents the potential outlier from escaping to $+ \infty$ and not appearing as a zero of the limit), the conclusion then follows from Montel's theorem and the argument principle. The latter ensures that the limit of the sets of zeros of the functions above on the considered neighborhood will converge to those of the left hand side of \eqref{theeq}.

\medskip
We have for $1 \leq a,b,c,d \leq L$ that
\begin{align*}
(T^*( z - \widetilde{\mathbf{X}}^{(N)})^{-1}T)_{ab,cd} &= (e_a \otimes u_b) ( z - \widetilde{\mathbf{X}}^{(N)})^{-1}  (e_c \otimes u_d)^*
\end{align*}
	
and therefore, using  invariance under orthogonal conjugation and concentration of measure, it follows
	
\[ \lim_{N \to \infty} (T^*( z - \widetilde{\mathbf{X}}^{(N)})^{-1}T)_{ab,cd} = - M(z)_{a,c} \Psi_{b,d} = - (M(z) \otimes \Psi)_{ab,cd},\] 
	
which proves our result.
\end{proof}

To apply Lemma~\ref{lem-theeq}, we need to ensure that~\eqref{theeq} indeed has a solution in the context we are considering. This is the content of the following proposition.

\begin{proposition}\label{prop-existencesol}
For every {$x>r_\infty$} and every profile $\Psi\in\mathrm{Sym}^{++,1}_L(\R)$, recalling the definition \eqref{eq-defphi} of $\phi$, there exists $\theta$ such that $x$ is the largest solution $Z(\theta)$ of the equation (in $z$)
\begin{equation}\label{theeq2}
\det(\Id_{L^2} +  2 \theta S ( M(z) \otimes \phi(\theta,x,\Psi)) ) = 0.
\end{equation} 
\end{proposition}

Let us fix such $\Psi$ and $\theta$ and define $Z(\theta)$ to be the largest solution of~\eqref{theeq2} if it has a solution strictly larger than $r_\infty$ or $r_\infty$ if it has no such solutions. 
{We first show that $ Z(\theta ) \leq c_0 + c_1 \theta$ for some finite constants $c_1,c_2$.}
First, note that we have 
\[ 2 \theta  \phi(\theta,x,\Psi) = (2 \theta + m_{\mu_\infty}(x))_+ \Psi + M(\max(-m_{\mu_\infty}^{-1}(2 \theta), x) ). \]
Hence, for $x>-m_{\mu_\infty}^{-1}(2\theta)$ it follows that
\[  2 \theta S ( M(z) \otimes \phi(\theta,x,\Psi)) = S (  M(z) \otimes ((2 \theta + m_{\mu_\infty}(x))_{+}\Psi + M(x))), \]
which yields the bound 
\begin{align*}
&\| S (  M(z) \otimes ((2 \theta + m_{\mu_\infty}(x))\Psi + M(x))) \|\\
&\leq \|S\|\cdot\|M(z)\|\cdot \| (2 \theta + m_{\mu_\infty}(x))_+ \Psi + M(\max(-m_{\mu_\infty}^{-1}(2 \theta), x) ) \|.
\end{align*} 
Since $\|M(z)\| \leq (z - r_\infty)^{ -1}$ we have that there exists some constants $c_0,c_1 >0$ such that for $ z > c_0 \theta + c_1 $, $\| S (  M(z) \otimes ((2 \theta + m_{\mu_\infty}(x))_+ \Psi + M(\max\{-m_{\mu_\infty}^{-1}(2 \theta), x\} ))\|< 1$ and, therefore, $ Z(\theta ) \leq c_0 + c_1 \theta$.

\medskip
To prove Proposition~\ref{prop-existencesol}, we further need that $Z$ is continuous.
\begin{lemma}
The function $Z$ is continuous on $\R^{+,*}$.
\end{lemma}

\begin{proof}
For $\mathcal{O}$ an open subset of  $\{ z \in \C: \Re z > r_\infty \}$, let us define $\mathcal{D}_{\mathcal{O}}$ the function from $\R^{+,*}$ to $\mathcal{H}(\mathcal{O})$ (the set of holomorphic functions on $\mathcal{O}$) as 
	
\[ \mathcal{D}_{\mathcal{O}}(\theta)(z) =  \det( \Id_{L^2} +  2 \theta S ( M(z) \otimes \phi(\theta,x,\Psi)) ). \]
	
Then $\mathcal{D}_{\mathcal{O}}$ is a continuous function (for the topology of uniform convergence on compact sets). First, let us see that it is sufficient to prove the following.

\medskip
\underline{\smash{Claim:}} For every $\epsilon >0$, the sets $Z^{ -1}( ] r_\infty+ \epsilon, + \infty[)$ are open and the function $Z$ is continuous on these sets.

\medskip
Indeed, this proves that $Z$ is continuous on the set on every point $\theta$ such that $Z( \theta ) > r$. Let us now take a point $\theta_0$ such that $Z( \theta_0 ) = r_\infty$. If we assume towards  contradiction that $\ell := \limsup_{\theta \to \theta_0} Z(\theta) > r_\infty$, then if $\ell < + \infty$ one can find a sequence $(\theta_n)_{n \in \N}$ such that $\lim_{n \to + \infty} \theta_n = \theta_0$ and $\lim_{n \to \infty} Z(\theta_n) = \ell$. Then, using the continuity of $\mathcal{D}_{\mathcal{O}}$ for an arbitrary open neighborhood $\mathcal{O}$ of $] r_\infty, + \infty[$, we have that $Z(\theta_0) = \ell$ which contradicts our assumption. Finally, note that  $\ell = + \infty$ is impossible since $Z(\theta)$ is bounded on every compact set of $\R^{+,*}$.

\medskip
Let us then move on to proving our claim. For $\theta \geq 0$, consider
\begin{displaymath}
z \mapsto \det( \Id_{L^2} +  2 \theta S ( M(z) \otimes \phi(\theta,x,\Psi)) )
\end{displaymath}
on an arbitrary open neighborhood $\mathcal{O}$ of $]r_\infty + \epsilon, + \infty[$, i.e., the function $\mathcal{D}_{\mathcal{O}}$. Recall that this function is the almost sure limit of
\begin{displaymath}
z \mapsto \frac{\det(z -\mathbf{X}^{(N)}) }{\det(z -\widetilde{\mathbf{X}}^{(N)})}
\end{displaymath}
as $N\rightarrow\infty$. Since the  zeroes of this sequence of functions   all lie on the real line (as a quotient of two polynomials with only real roots),  the same is true for its limit $\mathcal{D}_{\mathcal{O}}$ by using, e.g., the argument principle. Then, using the continuity of $\mathcal{D}_{\mathcal{O}}$ and once again using the argument principle, one concludes that the sets $Z^{ -1}( ] r_\infty+ \epsilon, + \infty[)$ are open. 
{Finally to prove the continuity of the function $Z$, we need only to remember that $Z(\theta)$ is bounded on every compact set of $\R^{+,*}$ which prevents there being a point $a \in \R^{+,*}$ such that $\limsup_{t \to a } Z(t) = + \infty$ and apply the argument principle one last time. }
\end{proof}

With these tools in place, we can give the proof of Proposition~\ref{prop-existencesol}.
\begin{proof}[Proof of Proposition~\ref{prop-existencesol}]
As a consequence of the continuity of $Z$, to prove that there exists~$\theta$ such that $Z(\theta)=x$ one only needs to notice that $Z(0)=r_\infty$, which is straightforward by Lemma \ref{lem-convlargestev}, and then to prove that there exists $\theta_0 >0$ such that $Z(\theta_0) > x$. For this, let us denote $\lambda(\theta,z)$ the largest eigenvalue of the matrix $- 2 \theta S ( M(z) \otimes \phi(\theta,x,\Psi))$. One can first notice that the matrices $S$, $M(z)$ and $\phi(\theta,x,\Psi)$ are symmetric matrices. Furthermore the matrices $-M(z)$ and $\phi(\theta,x,\Psi)$ are definite positive matrix and therefore so is $-  M(z) \otimes \phi(\theta,x,\Psi)$. Therefore, the matrix $- 2 \theta S ( M(z) \otimes \phi(\theta,x,\Psi))$ has the same spectrum as

\[ \sqrt{ - M(z) \otimes 2 \theta \phi(\theta,x,\Psi)} S \sqrt{ - M(z) \otimes 2 \theta \phi(\theta,x,\Psi)} \] 
which itself is symmetric. In particular,  $\lambda(\theta,z)$  is also the largest eigenvalue of the above matrix and the function $\theta \mapsto \lambda(\theta,z)$  is easily seen to be continuous. Furthermore, we have $\lambda(0,x) = 0$. To apply the intermediate value theorem, we further need the following.

\begin{lemma}\label{posdef}
For every $\Psi \in \mathrm{Sym}_L^{++,1}(\R) $, every $z>r_\infty$, 
$\lim_{\theta \to \infty} \lambda(\theta,z) = + \infty$.
\end{lemma}

We postpone the proof of Lemma~\ref{posdef} for the moment. Given the lemma, we conclude that, {since $\lambda(0,x)=0$ and $\lambda(.,x)$ is continuous,}   there exists some $\theta' >0$ such that $\lambda(\theta',x) = 1$. This implies that

\[ \det( \Id_{L^2} +  2 \theta' S ( M(x) \otimes \phi(\theta',x,\Psi)) )  =0,\]
which yields that indeed $Z(\theta') \geq x$. Then, using the fact that $Z(0) = r_{\infty}$  and the continuity of $Z$, we have by the intermediate value theorem that there exists $\theta \geq 0$ such that $Z( \theta) =x$. This finishes the proof of Proposition~\ref{prop-existencesol}.
\end{proof}

\medskip
It remains to supply the asymptotics of $\lambda(\theta,z)$ that we used above.

\begin{proof}[Proof of Lemma~\ref{posdef}]
First, one can notice that $\lambda_1(S) >0$. Indeed since all the $A_i$ are symmetric, so is~$S$. Furthermore, we have that 
\[ \Tr(S) = \sum_{i} \Tr(A_i \otimes A_i) = \sum_{i}\Tr(A_i)^2 \geq 0. \]
Therefore, $S$ has to have at least one positive eigenvalue. Since both matrices $ 2 \theta \phi(\theta,x,\Psi)$ and $-M(z)$ are definite positive, so is their tensor product and furthermore, we have that
\begin{align*}
&\lambda_1(\sqrt{ - M(z) \otimes 2 \theta \phi(\theta,x,\Psi)} S \sqrt{ - M(z) \otimes 2 \theta \phi(\theta,x,\Psi)}  )\\
&\geq {\lambda_{L^2}}(\sqrt{ - M(z) \otimes 2 \theta \phi(\theta,x,\Psi)})^2 \lambda_1(S).
\end{align*}
Note that $\lambda_{L^2}$ denotes the smallest eigenvalue here, as both $M(z)$ and $\phi(\theta,x,\Psi)$ are of size $L\times L$. Furthermore,
\[ \lambda_{L^2}(\sqrt{ - M(z) \otimes 2 \theta \phi(\theta,x,\Psi)}) = \sqrt{2 \theta \lambda_{L^2}( - M(z) \otimes  \phi(\theta,x,\Psi))}.
\]
For $2\theta \geq (- m_{\mu_{\infty}}(x))$, one has 

\[ 2 \theta \phi(\theta,x,\Psi) = -M(x)/L + (2\theta + m_{\mu_{\infty}}(x)) \Psi \]

Therefore, one can write

\[ \lambda_{L}(2 \theta \phi(\theta,x,\Psi)) \geq - ||M(x)||/L + (2\theta + m_{\mu_{\infty}}(x)) \lambda_{L}(\Psi) \] 

since by assumption $\lambda_{L}(\Psi)>0$, we have $\lim_{\theta \to \infty}\lambda_{L}(2 \theta \phi(\theta,x,\Psi)) = + \infty$ which proves the result.
\end{proof}

Putting together the preceding lemmas, we obtain the following.

\begin{proposition}\label{prop-impliesLB}
For any $\delta >0$, $x > r_{\infty}$,  $\Psi \in \mathrm{Sym}^{++,1}( \R)$ there exists $\theta>0$ so that 

\[\lim_{\epsilon \to 0}   \liminf_{N \to \infty} \inf_{\substack{\mathbf{u}\in\mathds{S}^{NL-1},\\ ||\rho(\mathbf{u})- \phi(\theta,x,\Psi)|| \leq \epsilon}} \P^{(\theta,u)}[ | \lambda_1(\mathbf{X}^{(N)}) - x | \leq \delta ] =1. \]
\end{proposition}

From this proposition, we can now deduce the desired weak LDP lower bound.

\begin{proof}[Proof of Proposition~\ref{prop-weakLDPinf}]

Restricting to $\rho(\mathbf{u})$ a neighborhood of $\Psi \in \mathrm{Sym}^{++,1}( \R)$, such that $\Tr[ \Psi^T S( \Psi)] \neq 0$, we obtain from Proposition~\ref{prop-impliesLB} that 

\[ \lim_{\delta \to 0} \liminf P^{(N)}_{\delta}(x) \geq - \inf_{\substack{\Psi \in \mathrm{Sym}^{++,1}(\R),\\ \Tr[\Psi^T \cS( \Psi)] \neq 0}} \sup_{\theta \geq 0} \mathcal{F}( \theta,x, \Psi). \]

To upgrade this statement to $\Psi$ that may not be definite, we can approximate any $\Psi\in \mathrm{Sym}^{+,1}(\R)$ so that $\Tr[\Psi^T \cS( \Psi)] \neq 0$ by a sequence $(\Psi_p)_{p \in \N}$ such that $\Psi_p \in \mathrm{Sym}^{++,1}_L(\R)$ and $\Psi_p \to \Psi$. For $p$ large enough we have that $\Tr[ \Psi_p^T \cS( \Psi_p)] \geq  \eta$ for some $\eta >0$. We then use Lemma \ref{lem-thetainacompact} to restrict $\theta$ to  a compact. From there the continuity of $\mathcal{F}( \theta,x, \Psi)$ in~$\Psi$ and~$\theta$ enables us to conclude. 
\end{proof}
Lastly, let us deduce a property of the rate function from the upper and lower bounds. This property shows that the two bounds actually match and will be useful to study the properties of $\rate$.

\begin{proposition}\label{prop-profileepsilon}
For $x > r_{\infty}$ and $\epsilon_x$ chosen as in Lemma \ref{prop-improved_weakLDPsup}, we have that for every $\epsilon \in (0,\epsilon_x]$, 
\[ \rate(x) = \rate(x, \epsilon). \] 
\end{proposition}
\begin{proof}
Clearly $\rate(x) \leq \rate(x,\epsilon) $. Furthermore using the LDP upper and lower bounds, more specifically Proposition \ref{prop-improved_weakLDPsup} for the upper bound, we have 

\[ \rate(x) \geq - \lim_{\delta \to 0} \liminf_{N\to + \infty}P^{(N)}_{\delta}(x) \geq  - \lim_{\delta \to 0} \limsup_{N\to + \infty}P^{(N)}_{\delta}(x) \geq \rate(x, \epsilon), \]
which proves the result.
\end{proof}

\section{Discussion of the Rate Function}\label{sect-ratefunct}
Lastly, we note the following properties of the rate function~\eqref{eq-ratefunct} in Theorem~\ref{thm-main}.
\begin{proposition}\label{prop-ratefunct}
The rate function $\rate$ in~\eqref{eq-ratefunct} satisfies the following properties:
\begin{itemize}
\item[(i)] $x\mapsto \rate(x)$ is continuous on $(r_\infty,\infty)$.
\item[(ii)] $x\mapsto \rate(x)$ is a good rate function, i.e., $\rate$ is lower semi-continuous and its sublevel sets $\{\rate(x)\leq y\}$ are compact for all $y\in\R$.
\item[(iii)] $x\mapsto \rate(x)$ is increasing on $[r_\infty,\infty)$.
\end{itemize}
\end{proposition}

\begin{remark} The rate function may be discontinuous at $r_{\infty}$. Indeed, let us consider for instance the matrix 

\[ \frac{1}{\sqrt{N}} \begin{pmatrix} G_N & 0_N \\0_N &  3 I_N \end{pmatrix} \]
where $G_N$ is a GOE matrix. Then the rate function for the largest eigenvalue of such a matrix would be $+ \infty$ for $x < 3$, $0$ for $x =3$ and $\rate(x)$ for $x >3$.
\end{remark}

Let us first {prove parts} (ii) and (iii) assuming (i) of the above proposition : 
\begin{proof}[Proof of (ii) and (iii) of Proposition \ref{prop-ratefunct}]

Beginning by (ii), we first notice that $\rate$ is lower semi-continuous at $r_\infty$ since it vanishes at $r_\infty$, and is non-negative everywhere. The level sets are included in compacts as a consequence of the existence of the LDP and the exponential tightness. We give another more direct proof using the expression of the rate function. Indeed, one can make the following bounds for  any $\Psi \in \mathrm{Sym}_L^{ +,1}(\R)$

\[ \Tr[ \Psi \mathcal{S}(\Psi) ] \leq A\,\quad \Tr[A_0 \Psi] \leq B,\quad 
\ln ( \det(\Psi) - L \ln L) \leq 0 \,,\]
for some constants $A,B > 0$ that depends only on the models. For $x > r_{\infty} +1$, we can write with $\theta_*= - m_{\mu_{\infty}}( r_{\infty} +1)$ for $\theta > \theta_*$
\[ J_{\mu_{\infty}}(\theta,x) = \theta x - \frac{1}{2}( 1 + \ln(2 \theta)) - \frac{1}{2}\int_{\R} \ln |x-y| d \mu_{\infty}(y) \geq \theta x - \frac{1}{2}( 1 + \ln(2 \theta)) - \frac{1}{2}\ln|x| - C \]
for some constant $C$. 
Therefore we have for $x > r_{\infty} +1$,

\begin{eqnarray*}
 \rate(x) &=& \inf_{\Psi} \sup_{\theta \geq 0} \mathcal{F}( \theta,x,\Psi)  \\
 & \geq &  \inf_{\Psi} \sup_{\theta \geq \theta_*} \mathcal{F}( \theta,x,\Psi) \\
  & \geq & \sup_{\theta \geq \theta_*}\Big[  L( \theta x - \frac{1}{2}( 1 + \ln(2 \theta)) - \frac{1}{2}\ln|x| - C) - L^2 A \theta^2 - L B \theta  \Big]\\
  & \geq & \Big[  L( \theta_* x - \frac{1}{2}( 1 + \ln(2 \theta_*)) - \frac{1}{2}\ln|x| - C) - L^2 A \theta_*^2 - L B \theta_*  \Big]
  \end{eqnarray*}

It is now easy to see that the lower bound goes to $+ \infty$ as $x$ goes to $+ \infty$ which proves the result. 

\medskip
For point (iii) of the proposition, one can prove that for any $x,y$ such that $r_{\infty} < y <x $ and any $\delta >0$ such that $\delta < |y - r_{\infty}|$, we have

\[ \limsup_{N \to + \infty}P_{\delta}^{(N)}(y ) \leq  \limsup_{N \to + \infty}P_{\delta}^{(N)}(x). \]
For this, one can reproduce almost verbatim the proof of Lemma B.1 in \cite{BoursierG24}, the only difference is that the matrix valued stochastic process used has to be defined by running an Ornstein-Uhlenbeck on every matrix $W_j^{(N)}$. In other words, in the proof, instead of~$G_N$, we consider the process $\mathbf{X}_\beta^{(N)}(t)$ defined by 
\[ \mathbf{X}_\beta^{(N)}(t) = \sum_{j=1}^k A_j\otimes W_j^{(N)}(t)+A_0\otimes \Id_N\]
where $W_j^{(N)}$ obey the following stochastic differential equations with $(H_j^{(N)})_{ 1 \leq j \leq k}$ being a family of $k$ i.i.d. $N \times N$ symmetric Brownian motions

\[ \forall 1 \leq j \leq k, dW_j^{(N)} = d H_j^{(N)} - \frac{1}{2} W_j^{(N)}. \]
The rest of the proof is identical. From the proved inequality, taking $\delta$ to $0$ gives that $\rate(y) \leq \rate(x)$. 
 \end{proof}

To discuss the continuity of the rate function, we will need to restrict ourselves to a compact interval of $\theta \geq 0$
\begin{lemma}\label{lem-thetainacompact}
Let  $M > \eta  > r_{\infty} $ and $\eps >0$. There exist $\Theta(M,\eta,\epsilon)$ such that for every $\Psi$ such that if  $\Tr[ \Psi^T S(\Psi)] \geq \eps$, we have for every $x \in (r_{\infty},M]$
\[ \sup_{\theta \geq 0} \mathcal{F}(\theta,x,\Psi) = \sup_{0 \leq \theta \leq \Theta(M,\eta,\eps) } \mathcal{F}(\theta,x,\Psi) \]
\end{lemma}

We delay the proof of Lemma~\ref{lem-thetainacompact} to Appendix \ref{sect-functFpart1}.

\begin{proof}[Proof of (i) of Proposition \ref{prop-ratefunct}]
Let us consider $\eta,M \in \R$ such that $ r_{\infty} < \eta <M $, and let $\epsilon' := \inf_{x \in[r,M] \epsilon_x}$ as defined in Proposition \ref{prop-improved_weakLDPsup}. We have $\epsilon' >0$. Using the preceding Lemma, we have 
\[ \sup_{\theta \geq 0} \mathcal{F}(\theta,x,\Psi,1) = \sup_{0 \leq \theta \leq \Theta(M,\eta,\eps') } \mathcal{F}(\theta,x,\Psi) \] 
and, using Proposition \ref{prop-profileepsilon},
\[ \rate(x) = \inf_{\Psi : \langle \Psi,\mathcal{S}[\Psi]\rangle \geq \epsilon'} \sup_{0 \leq \theta \leq \Theta(M,\eta,\eps') } \mathcal{F}(\theta,x,\Psi). \]

Furthermore, the function 
\[ (\theta,x,\Psi) \mapsto \mathcal{F}(\theta,x,\Psi)\] 
is uniformly continuous on $[0, \Theta(M,\eta,\epsilon')] \times [\eta,M] \times \{ \Psi : \langle \Psi, \mathcal{S} \Psi \rangle \geq \epsilon'\}$. This implies that the function $\rate$ is continuous on $[\eta,M]$ and therefore on $]r_{\infty}, + \infty[$.
\end{proof}

\section{Changes for the GUE}\label{sect-GUE}
This section collects several remarks on the differences between the proof of Theorem~\ref{thm-main} for $\beta=1$ and $\beta=2$. While the overall argument is the same, some small adjustments may be needed to adapt the proof to the complex model. We collect the most notable ones below. Recall that for $\beta=2$, we have
\begin{displaymath}
\mathbf{X}^{(N)}:=\sum_{j=1}^k A_j\otimes W_j+A_0\otimes \Id_N
\end{displaymath}
with Hermitian $A_0,A_1,\dots,A_k\in\C^{L\times L}$ and $W_1,\dots,W_k$ independent GUE matrices of size~${N\times N}$.

\subsection{Proof of Proposition~\ref{prop-conv.annealed} for $\beta=2$}
When considering the convergence of the annealed spherical integral, the main difference between the real and complex case lies in evaluating $\E_W[\re^{\theta NL\langle\mathbf{u},\mathbf{X}^{(N)}\mathbf{u}\rangle}]$. Recalling our convention that complex inner products are skew-linear in the first entry, we write
\begin{align}
\langle\mathbf{u},(A_j\otimes W_j)\mathbf{u}\rangle&=\sum_{a=1}^N(W_j)_{aa}\sum_{c,d=1}^L(A_j)_{cd}\overline{(u_c)_a}(u_d)_b\NN\\
&\quad+\sum_{a<b}\Re(W_j)_{ab}\sum_{c,d=1}^L\big[(A_j)_{cd}\overline{(u_c)_a}(u_d)_b+\overline{(A_j)_{cd}}(u_c)_a\overline{(u_d)_b}\, \big]\label{eq-decompcomplex}\\
&\quad+\sum_{a<b}\ri \Im(W_j)_{ab}\sum_{c,d=1}^L\big[(A_j)_{cd}\overline{(u_c)_a}(u_d)_b-\overline{(A_j)_{cd}}(u_c)_a\overline{(u_d)_b}\, \big].\NN
\end{align}
By definition of the GUE, the $(W_j)_{aa}$ are i.i.d. real-valued $\mathcal{N}(0,N^{-1})$ random variables and the $\Re(W_j)_{ab}$ resp. $\Im(W_j)_{ab}$ are i.i.d real-valued $\mathcal{N}(0,(2N)^{-1})$ random variables for $a<b$. Moreover, these three families of Gaussian random variables are independent of each other. Noting that for $\xi\sim\mathcal{N}(0,\sigma^2)$, the formula $\E[\re^{t\xi}]=\re^{t^2/(2\sigma^2)}$ at the base of~\eqref{eq-defL} continues to hold for $t\in\C$,~\eqref{eq-decompcomplex} implies that
\begin{align*}
\prod_{j=1}^k\E_W[\re^{\theta NL\langle\mathbf{u},(A_j\otimes W_j)\mathbf{u}\rangle}]&=\prod_{j=1}^k\exp\Big[\sum_{a=1}^N\frac{\theta^2N^2L^2}{2N}\Big(\sum_{c,d=1}^L(A_j)_{cd}\overline{(u_c)_a}(u_d)_b\Big)^2\\
&\quad\quad+\sum_{a<b}\frac{\theta^2N^2L^2}{4N}\Big(\sum_{c,d=1}^L\big[(A_j)_{cd}\overline{(u_c)_a}(u_d)_b+\overline{(A_j)_{cd}}(u_c)_a\overline{(u_d)_b}\Big)^2\\
&\quad\quad-\sum_{a<b}\frac{\theta^2N^2L^2}{4N}\Big(\sum_{c,d=1}^L\big[(A_j)_{cd}\overline{(u_c)_a}(u_d)_b-\overline{(A_j)_{cd}}(u_c)_a\overline{(u_d)_b}\Big)^2\Big]\\
&=\prod_{j=1}^k\exp\Big[\frac{NL^2\theta^2}{2}\Tr[\rho(\mathbf{u})^TA_j\rho(\mathbf{u})^TA_j]\Big].
\end{align*}
Similar to~\eqref{eq-innerprod}, the dependence of $\E_W[\exp(\theta NL\langle\mathbf{u},\mathbf{X}^{(N)}\mathbf{u}\rangle)]$ on $\mathbf{u}$ is thus completely determined by $\rho(\mathbf{u})_{ij}=\langle u_i,u_j\rangle$, i.e., the entries of a (complex) renormalized Wishart matrix. Moreover, these random variables again satisfy an LDP and the rate function for this complex analog of Lemma~\ref{lem-WishartLDP} is given by
\begin{displaymath}
\widetilde{\rate}:\mathrm{Sym}_L^{+,1}(\C)\rightarrow\C,\quad \widetilde{\rate}(X)=L\ln(L)-\ln(\det(X))
\end{displaymath}
The proof of the LDP is analogous to that of Lemma~\ref{lem-WishartLDP}. However, instead of~\cite{StatBook}, we now use Equation (1.6) of~\cite{Goodman1963} as the starting point, which indicates that the joint law of the~(upper triangular) entries of a complex Wishart matrix is given by
\begin{displaymath}
f_Y(X)={\frac {1}{\pi ^{L(L-1)/2}\prod _{j=1}^{L}\Gamma_{\C}(N-j-1)}}{\mathrm{det}(X)}^{(N-L)}\mathrm {e} ^{-\Tr (X)},\quad  X\in\mathrm{Sym}_L^{++}(\C)
\end{displaymath}
where $\Gamma_{\C}$ denotes the complex gamma function.

\subsection{Proof of Proposition~\ref{prop-weakLDPsup} for $\beta=2$}
For the weak LDP upper bound, we observe that~\eqref{eq-defmuproj} may be a complex measure, as the coefficients of the Dirac measures in
\begin{displaymath}
\mu_{\mathbf{H}}(\Pi_i,\Pi_j)=\frac{1}{N}\sum_{a=1}^{NL}\langle \mathbf{P}^{i\leftrightarrow j}\Pi_i\mathbf{v}_a(\mathbf{H}),\Pi_j\mathbf{v}_a(\mathbf{H})\rangle \delta_{\lambda_a(\mathbf{H})}
\end{displaymath}
are given by complex inner products. However, the bound~\eqref{eq-complexRadon} still holds, hence $\mu_{\mathbf{H}}(\Pi_i,\Pi_j)$ is thus well-defined as a complex Radon measure and $(\mu_{\mathbf{H}}(\Pi_i,\Pi_j))_{ij}$ can be interpreted as a random measure taking values in $\mathrm{Sym}_L^+(\C)$. Hence, considering
\begin{displaymath}
d_{Lip}(\nu_1,\nu_2)=\sup_{\substack{\|f\|_\infty\leq1,\\ \mathrm{Lip}(f)\leq1}}\Big|\int f\dx(\nu_1-\nu_2)\Big|
\end{displaymath}
and decomposing the complex measures into their real and imaginary parts allows reducing the argument to signed measures. The rest of the proof is then analogous to the real case. Note, however, that we need to use \cite[Thm.~4.5]{HaagerupTobjornsen2005} instead of~\cite[Thm.~4.3]{HST2006} in the proof of Theorem~\ref{thm-muproj.conc}, as the randomness in $\mathbf{X}^{(N)}_\beta$ is supplied by GUE matrices for~$\beta=2$.

\appendix
\section{Additional Proofs}
\subsection{Proof of Lemma~\ref{lem-WishartLDP} (LDP for $\rho(\mathbf{u})$)}\label{app-WishartLDP}

We follow the standard technique of proving a weak LDP and exponential tightness to obtain the desired result. Recall that $\mathrm{Sym}_L^+(\R)$ resp. $\mathrm{Sym}_L^{++}(\R)$ denotes the set of symmetric positive semi-definite resp. symmetric positive definite matrices, and
\begin{displaymath}
\mathrm{Sym}_L^{\sharp,1}(\R)=\{X\in\mathrm{Sym}_L^\sharp(\R):\Tr(X)=1\}
\end{displaymath}
for $\sharp\in\{+,++\}$. Moreover, $\rho(\mathbf{u})=\widetilde{Y}$ was defined in~\eqref{eq-defYtilde}.

\medskip
\underline{\smash{Exponential tightness:}}
As $\mathrm{supp}(\rho(\mathbf{u}))=\mathrm{Sym}_L^{+,1}(\R)$ is a compact set, exponential tightness is trivially satisfied. 

\medskip
\underline{\smash{The law of $\rho(\mathbf{u})$:}}  As $Y$ is a Wishart matrix, the joint law of the (upper triangular) entries of~$Y$ is absolutely continuous with density given by
\begin{displaymath}
f_Y(X)={\frac {1}{2^{\frac {NL}{2}}\pi ^{L(L-1)/4}\prod _{j=1}^{L}\Gamma(\frac {N-j-1}{2})}}{\mathrm{det}(X)}^{\frac {N-L-1}{2}}\mathrm {e} ^{-{\frac {1}{2}}{\Tr (X)}},\quad X\in\mathrm{Sym}_L^{++}(\R)
\end{displaymath}
where $\Gamma(\cdot)$ denotes the gamma function (see, e.g.,~\cite[Sect.~7.2]{StatBook}). In this notation, we make the change of variables
\begin{displaymath}
X\mapsto\Big(\frac{X}{\Tr(X)},\Tr(X)\Big)=:(\widetilde{X},t),
\end{displaymath}
and note that the entries of the Jacobian $J$ are given by
\begin{align*}
\frac{\partial \widetilde{X}_{ij}}{\partial X_{kl}}=\frac{1}{t}(\I_{\{i=k\}}\I_{\{j=l\}}-\I_{\{k=l\}}\widetilde{X}_{ij}),\quad \frac{\partial t}{\partial X_{kl}}=\I_{\{k=l\}}.
\end{align*}
The determinant is easily computed using Gaussian elimination, yielding
\begin{displaymath}
\det(J)=t^{1-\frac{L(L+1)}{2}}.
\end{displaymath}
After the change of variables, we obtain
\begin{align*}
f_{Y}(\widetilde{X},t)=\frac {{\mathrm{det}(\widetilde{X})}^{\frac {N-L-1}{2}}t^{\frac{NL-2L^2-2L+2}{2}}\mathrm {e} ^{-{\frac {t}{2}}}}{2^{\frac {NL}{2}}\pi ^{L(L-1)/4}\prod _{j=1}^{L}\Gamma(\frac {N-j-1}{2})},\quad (\widetilde{X},t)\in\mathrm{Sym}_L^{++,1}(\R)\times\R_+
\end{align*}
which factorizes into two parts that only depend on $\widetilde{X}$ and $t$, respectively. Integrating out the $t$ variable yields the density for $\rho(\mathbf{u})$. It is given by
\begin{equation}\label{eq-lawYtilde}
f_{\widetilde{Y}}(X)=\frac {\Gamma(\frac{NL}{2}-L^2-L+2)}{2^{L(L+1)+2}\pi ^{L(L-1)/4}\prod _{j=1}^{L}\Gamma(\frac {N-j-1}{2})}{\mathrm{det}(X)}^{\frac {N-L-1}{2}},\quad X\in\mathrm{Sym}_L^{++,1}(\R)
\end{equation}
For simplicity, we abbreviate the normalizing constant in~\eqref{eq-lawYtilde} as $Z_N$.

\medskip
\underline{\smash{Weak LDP:}} We start by writing
\begin{displaymath}
Z_N=\int_{\mathrm{Sym}_L^{+,1}(\R)}\exp\Big[N\ln(\det(X))\Big(\frac{1}{2}-\frac{L-1}{2N}\Big)\Big]\dx X
\end{displaymath}
and noting that the function in the exponent satisfies
\begin{displaymath}
\lim_{N\rightarrow\infty}\Big[\ln(\det(X))\Big(\frac{1}{2}-\frac{L-1}{2N}\Big)\Big]=\frac{\ln(\det(X))}{2}
\end{displaymath}
whenever $X$ does not have a zero eigenvalue (otherwise, the contribution to the integral vanishes). Moreover, the convergence is uniform on all sets where all eigenvalues of $X$ are bounded away from zero, and for every constant $C > 0$ there is an $\eps > 0$ such that
\begin{displaymath}
\ln(\det(X)) < -C
\end{displaymath}
for all $X$ for which at least one eigenvalue is smaller than $\eps$. With these observations, we deduce from the classical Laplace method that
\begin{displaymath}
\lim_{N\rightarrow\infty}\frac{1}{N}\ln(Z_N)=\max_{X\in\mathrm{Sym}_L^{+,1}(\R)}\frac{\ln(\det(X))}{2}=-\frac{L\ln(L)}{2}.
\end{displaymath}
Here, the last equality follows by interpreting the optimization problem in terms of the~(non-negative) eigenvalues $\lambda_1,\dots,\lambda_L$ of $X\in\mathrm{Sym}_L^{+,1}(\R)$ and maximizing $\det(X)=\prod_i\lambda_i$ under the constraint $\Tr(X)=\sum_i\lambda_i=1$. This maximum is attained for $\lambda_1=\dots=\lambda_L=1/L$, which yields the desired constant.

\medskip
Another application of the Laplace method gives the desired weak LDP with rate function $\widetilde{\rate}(X)=(\ln(\det(X))-L\ln(L))/2$. Note that we defined $\widetilde{\rate}$ for positive semi-definite matrices instead of the positive definite matrices to ensure we work on a Polish space. However, if $\|\rho(\mathbf{u})-X\|<\eps$ for $\eps>0$ and a matrix $X$ that has a zero eigenvalue, $\rho(\mathbf{u})$ must have at least one small eigenvalue and $\det(\rho(\mathbf{u}))<\eps$. As a consequence, $\widetilde{\rate}(X)=\infty$ as expected.

\subsection{The Function $\cF$: Proof of Lemmas \ref{lem-Ibound} and \ref{lem-thetainacompact}}\label{sect-functFpart1}
First, let us remind that $\mathcal{F}(\theta,x,\Psi)=0$ for every $\theta \leq \theta_x:=-m(x) $. Then, let us call for a matrix $M$, $\Delta(M) = - \frac{1}{2}( \ln \det(M) - L \ln L )$. Let us then differentiate the function $\mathcal{F}$ with respect to $\theta$ for $\theta \geq \theta_x$, we get 

\begin{eqnarray*}
\frac{\partial \mathcal{F}}{\partial \theta}(\theta,x,\Psi) &=& L( x -  \frac{1}{2\theta}) -2a(\theta -\theta_x) - 2b + \frac{\partial\Delta\circ \phi}{\partial\theta}(\theta,x,\Psi), \\
\end{eqnarray*}
where 
\[ a:= \frac{L^2}{4}\Tr[ \Psi \mathcal{S}(\Psi) ] \text{ and } b:= 2L\Tr( \Psi A_0^T) + \frac{L^2}{2}\Tr[ (-M(x)) \mathcal{S}(\Psi) ]. \]
Let us study the function 
\[ g: \theta \mapsto - \frac{L}{2\theta} + \frac{\partial\Delta\circ \phi}{\partial\theta}(\theta,x,\Psi). \]
First, we can see that $- L/2\theta$ is the derivative of the function $\theta \mapsto - \frac{1}{2}\ln \det(2\theta \Id_L)$. therefore, the function $g$ is the derivative of the function 

\[ G : \theta \mapsto -\frac{1}{2}\Big( \ln \det(2 \theta\phi(\theta,x,\Psi)) - L \ln L  \Big) = \Delta( 2 \theta\phi(\theta,x,\Psi))\]
Let us now notice that for $\theta \geq \theta_x$, $2\theta\phi(\theta,x,\Psi) = ( - M(x) + (2 (\theta - \theta_x)) \Psi)$ and let us remind that the differential of $\Delta$ in an invertible matrix $A$ is given by $d\Delta_A(H) = -\frac{1}{2}\Tr(A^{-1}H)$. 
Using the chain rule, we can rewrite $g$ as 
 \[  g(\theta) = - \Tr( \Psi ( - M(x) + (2 (\theta - \theta_x)) \Psi)^{-1}).  \]
Since $-M(x) $ and $\Psi$ are positive matrices, the above expression is negative. Furthermore, we have 
$\Tr[ (-M(x)) \mathcal{S}(\Psi) ] \geq 0 $ and we can bound $2 L \Tr(\Psi A_0^T)$ from below by some constant $- b_0$ with $b_0 >0$. All in all, we can write 
\begin{eqnarray*}
\frac{\partial \mathcal{F}}{\partial \theta}(\theta,x,\Psi) &\leq &  (x + b_0) - 2a (\theta- \theta_x),
\end{eqnarray*}
and thus  since $\mathcal{F}(\theta_x,x,\Psi)=0$, for $\theta\ge \theta_x$, 
\begin{eqnarray*}
\mathcal{F}(\theta,x,\Psi) &\leq &  (x + b_0)(\theta-\theta_x) - a (\theta- \theta_x)^2.
\end{eqnarray*}
From this equation, we deduce that 
\[ \max_{\theta \geq 0}\mathcal{F}(\theta,x,\Psi) \leq \frac{(x+b_0)^2}{4a} \]
and that for $\theta \geq \theta_x + \frac{x+b_0}{a}$, it holds that $\mathcal{F}(\theta,x,\Psi) \leq 0$. Now, to prove Lemma \ref{lem-Ibound}, one can first choose $\Psi= \Id_L/L$ and $\eta = \Tr[ S( \Id_L/L)]$. We then have that for $\epsilon < \eta$
\begin{align*}
\rate(\sigma, \epsilon) &\leq \max_{\theta \geq 0} \mathcal{F}(\theta,x,\Id_L /L)\leq\frac{(x+b_0)^2}{4a} \leq \frac{x^2 + b_0^2}{2 a},
\end{align*}
and the lemma is proven choosing $\gamma = \max \{ 1,b_0^2 \}/2a$. To prove Lemma~\ref{lem-thetainacompact}, we can choose 
\[ \Theta(M, \eta, \eps) = -m_{\mu_{\infty}}(\eta) + 4\frac{M+b_0}{L^2 \eps}, \]
which then ensures the function $\mathcal{F}(\theta,x,\Psi)$ is negative for $\theta \geq \Theta(M,\eta,\eps)$.

\renewcommand*{\bibname}{References}

\interlinepenalty=10000

\bibliographystyle{plain}
\bibliography{LDP_references}

\end{document}